\documentclass[a4paper]{amsart}

\usepackage{lineno,hyperref,amsmath, amssymb,amsthm}
\usepackage[dvipdfmx]{graphicx}
\usepackage{color}
\usepackage{bm}

\newtheorem{theorem}{Theorem}[section] 
\newtheorem{lemma}[theorem]{Lemma}
\newtheorem{remark}[theorem]{Remark}
\newtheorem{proposition}[theorem]{Proposition}
\newtheorem{corollary}[theorem]{Corollary}

\newcommand{\C}{\mathbb{C}}
\newcommand{\R}{\mathbb{R}}
\newcommand{\K}{\mathbb{K}}

\newcommand{\p}{\partial}

\newcommand{\onab}{\overline{\nabla}}

\newcommand{\fg}{\mathfrak{g}}

\newcommand{\fk}{\mathfrak{k}}
\newcommand{\fm}{\mathfrak{m}}
\newcommand{\fn}{\mathfrak{n}}

\begin{document}

\title[Index estimate by first Betti number of minimal hypersurfaces]
{Index estimate by first Betti number of minimal hypersurfaces in compact symmetric spaces}

\author[T.Kajigaya, K. Kunikawa]{Toru Kajigaya and Keita Kunikawa}
\address[T. Kajigaya]{Department of Mathematics, Faculty of Science, Tokyo University of Science
1-3 Kagurazaka Shinjuku-ku, Tokyo 162-8601 JAPAN}
\email{kajigaya@rs.tus.ac.jp}

\address[K. Kunikawa]{Graduate School of Technology, Industrial and Social Sciences, Tokushima University, 2-1 Minamijosanjima-cho,
Tokushima, 770-8506, Japan}
\email{kunikawa@tokushima-u.ac.jp}




\date{\today}

\begin{abstract}
We show that the Morse index of unstable closed minimal hypersurface $\Sigma$ in a compact semi-simple Riemannian symmetric space $M=G/K$ is bounded from below by constant times the first Betti number of $\Sigma$. Our proof is based on a natural extension of the previous method and this also provides a novel approach for the index estimate. 
 \end{abstract}

\maketitle

\section{Introduction}
Let $\Sigma$ be a closed hypersurface in a closed Riemannian manifold $(M,g)$. $\Sigma$ is called {\it minimal} if it is a critical point of the volume functional under any infinitesimal deformation of $\Sigma$. The study of minimal hypersurfaces is still actively researched in connection with various areas of mathematics. One of the main concern is related to the Morse index and the topology of minimal hypersurface. The notion of {\it index} is defined by the maximal dimension of the subspace of the space of normal vector fields such that the Hessian of the volume functional is negative-definite on the subspace, and this measures how many directions exist to decrease the volume of $\Sigma$. There are many interesting studies on the estimate of index and the relation to geometry of minimal hypersurface in a given Riemannian manifold. We refer the reader to \cite{ACS, Mar, Nev, Ros, Song} and references therein for a detailed background and related results. 

As stated in \cite{ACS} (see also \cite{GMR, Mar, MR, Nev, Song}), there is a conjecture by Schoen, Marques and Neves that, if $M$ has positive Ricci curvature and ${\rm dim}M\geq 3$, then there exists a positive constant $C$ depending only on $M$ such that the inequality 
\begin{align}\label{eq:SMN}
{\rm index}(\Sigma)\geq C b_1(\Sigma)
\end{align}
holds for any closed embedded minimal hypersurface $\Sigma$ in $M$, where $b_1(\Sigma)$ is the first Betti number of $\Sigma$. This conjecture was confirmed by Savo \cite{Savo} in the case when $M$ is the standard sphere  $S^n$, and by generalizing his method, Ambrozio-Carlotto-Sharp \cite{ACS} proved the conjecture affirmatively when $M$ is a compact rank-one symmetric space (i.e. $M$ is either $S^n$, $\R P^n$, $\C P^n$, $\mathbb{H} P^n$ or the Cayley plane $\mathbb{O} P^2$). It is a natural attempt to extend these results to a general compact Riemannian symmetric space ({\it RSS} for short) of positive Ricci curvature, in particular, to a compact {\it semi-simple} RSS.  Here, a compact symmetric space is called {semi-simple} if the identity component of the isometry group $I_{0}(M)$ becomes a compact semi-simple Lie group, and it provides a very typical example of Riemannian manifold of positive Ricci curvature.

 In this direction,  by using the same method,  Ambrozio-Carlotto-Sharp \cite{ACS} and Gorodski-Mendes-Radeschi \cite{GMR} verified the conjecture for the following compact semi-simple symmetric spaces;  The product of sphere $S^p\times S^q$ of $p,q\geq 2$ and $(p,q)\neq (2,2)$, the quaternionic Grassmannian manifolds, $Sp(n)$ (for any $n$) and $SU(n)$ of $n\leq 17$. We remark that, besides these symmetric examples, the conjecture has been also confirmed  under some situations (see \cite{ACS, GMR} for other results).  However, as considered in \cite{GMR}, it seems difficult to prove the conjecture by using the method in \cite{ACS} for a general compact symmetric space.

In \cite{MR}, Mendes-Radeschi generalized the method by introducing the notion of {\it virtual immersion} of a Riemannian manifold $M$.   Moreover, they proved that the inequality \eqref{eq:SMN} holds when $M$ is a product of two compact rank-one symmetric spaces, or for an oriented minimal hypersurface $\Sigma$ in an oriented compact semi-simple symmetric space with an additional assumption that there exists a point $p\in \Sigma$ such that the principal curvatures at $p$ are distinct (\cite[Corollary D]{MR}).  
Note that as mentioned in \cite{MR}, the last result of \cite{MR} can be slightly extended,  namely, the same conclusion holds  if $\Sigma$ is unstable (i.e. ${\rm index}(\Sigma)\geq 1$) and satisfies the assumption of principal curvatures.  See also Appendix \ref{A:MR} for the method using the virtual immersion.

The aim of the present paper is to provide another approach to extend the method used by Savo and Ambrozio-Carlotto-Sharp. Furthermore, we confirm that the inequality \eqref{eq:SMN} holds if $M$ is a compact semi-simple RSS and the minimal hypersurface $\Sigma$ in $M$ is unstable.  More precisely, we show the following result.

\begin{theorem}[Theorem \ref{thm:main}]\label{thm:main0}
Let $M$ be a compact semi-simple Riemannian symmetric space and $\Sigma$ be a closed minimal hypersurface embedded in $M$. If $\Sigma$ is unstable, then it holds that
\begin{align}\label{eq:main0}
{\rm index}(\Sigma)\geq \frac{2}{d(d-1)+2(2n-3)}b_1(\Sigma),
\end{align}
where  $d$ is the dimension of  the isometry group  of $M$ and $n={\rm dim} M$.
\end{theorem}

An embedded hypersurface $\Sigma$ is called {\it two-sided} if the normal bundle is trivial in $M$ (otherwise, $\Sigma$ is called {\it one-sided}). It follows from the second variational formula (cf. \cite{Simons}) that any two-sided minimal hypersurface $\Sigma$ is unstable if $M$ has positive Ricci curvature. Therefore, we obtain the following corollary.

\begin{corollary}\label{cor:main1}
The inequality \eqref{eq:main0} holds  for every two-sided closed minimal hypersurface $\Sigma$ embedded in  a compact semi-simple Riemannian symmetric space.
\end{corollary}

If $M$ is simply-connected, then any embedded hypersurface in $M$ is two-sided. Thus,  we obtain 

\begin{corollary}
The conjecture by Marques-Neves-Schoen holds for any simply-connected compact semi-simple Riemannian symmetric space. 
\end{corollary}

We shall make some remarks on our results. Firstly, in Theorem \ref{thm:main0}, we proved an index bound  under the assumption that $\Sigma$ is unstable (Note that Theorem \ref{thm:main0} holds even if $\Sigma$ is one-sided). Therefore in a compact semi-simple RSS, the problem is still remaining for stable minimal hypersurfaces.  Notice that  a stable minimal hypersurface in a compact semi-simple RSS $M$ can exist only when $M$ is not simply-connected and $\Sigma$ is one-sided in $M$. Moreover, the inequality \eqref{eq:SMN} suggests that $b_{1}(\Sigma)$ must be $0$ if $\Sigma$ is stable. For example, the totally geodesic $\R P^{n-1}$ is stable in $\R P^{n}$, and when $M=\R P^n$, the inequality \eqref{eq:SMN} has been proved for any embedded closed minimal hypersurfaces (\cite{ACS}). Secondly, we remark that in  (2),  the constant depending on $M$   is not necessarily better than the previous results given in \cite{ACS, GMR, MR, Savo}. For example, the result by Ambrozio-Carlotto-Sharp \cite{ACS} provides a much better constant when $M$ is a compact rank-one symmetric space, and Mendes-Radeschi \cite{MR} proved more precise result when $\Sigma$ has a point $p$ such that the principal curvatures at $p$ are distinct.


The method by Ambrozio-Carlotto-Sharp \cite{ACS} is based on a kind of ``averaging method", which goes back to the method used in \cite{LS, Simons}. More precisely, the method of \cite{ACS} is summarized as follows. Let $\Sigma$ be a closed minimal hypersurface in a Riemannain manifold $M$.  First, we generate a test normal variation $\Phi_{\omega}(v)$ of  $\Sigma$ in $M$ by combining a harmonic $1$-form $\omega$ and  a vector $v\in \mathcal{G}$ of a finite dimensional metric vector space $\mathcal{G}$, namely, we take a bilinear map
\[
\Phi: \mathcal{H}\times\mathcal{G}\to \Gamma(\nu\Sigma),
\]
where $\mathcal{H}$ is the space of harmonic $1$-forms on $\Sigma$ and $\Gamma(\nu\Sigma)$ denotes the space of normal vector fields along $\Sigma$ in $M$. Next, we consider a quadratic form $q_{\omega}(v,w):=Q(\Phi_{\omega}(v),\Phi_{\omega}(w))$ on  $\mathcal{G}$, where $Q$ is the index form of the minimal hypersurface $\Sigma$ in $M$. Then it turns out that a sufficient condition for the inequality $\eqref{eq:SMN}$ is given by ${\rm Tr}q_{\omega}<0$ for any non-trivial $\omega\in \mathcal{H}$ (see Section \ref{sec:ind} for details). A problem here is how to define a preferred test variation $\Phi$ in order to estimate the trace of $q_{\omega}$. In \cite{ACS}, two reasonable choices of  test variations are presented by using an isometric immersion of the ambient Riemannian manifold $M^{n}$ into the Euclidean space $\R^{n+k}$. The first method is inspired by the idea of Ros \cite{Ros} and Urbano \cite{Urbano},  and this is applicable to a minimal surface in a $3$-dimensional Riemannian manifold (see \cite[Proposition 4]{ACS}). Another choice is a generalization of the method by Savo \cite{Savo}, and this can be applied to a minimal hypersurface of any dimension (\cite[Theorem A]{ACS}). In any case, an advantage of these methods is that the trace formula for $q_{\omega}$ is derived in a better form so that ${\rm Tr}q_{\omega}$ can be estimated by using  geometry of the isometric immersion $M\to \R^{n+k}$ (See \cite{ACS, GMR} or Appendix for details). 
Thus, if furthermore one can find a preferred isometric immersion of $M$ into $\R^{n+k}$ so that ${\rm Tr}q_{\omega}<0$ for any non-trivial $\omega\in \mathcal{H}$ on any $\Sigma\subset M$, we obtain the index estimate \eqref{eq:SMN} for any $\Sigma\subset M$. However, it should be emphasized that whether this method works depends on the choice of isometric immersion $M\to \R^{n+k}$.

In the present paper, we show that the methods in \cite{ACS} are naturally extended by considering an isometric immersion of a Riemannian manifold $M$ into a compact semi-simple RSS $N=G/K$. Indeed, we can define the  test  variation $\Phi$ and a quadratic form $q_{\omega}$ by a similar way via the isometric immersion $M\to N=G/K$ (See Section \ref{sec:tr1} and \ref{sec:tr2}).  Moreover, we obtain generalized trace formulas (Proposition \ref{prop:tr1} and Theorem \ref{thm:tr}) as extensions of formulas given in \cite[Proposition 1 and 2]{ACS}. 
We remark that if $M=N=G/K$, it turns out that the test variation given in Section \ref{sec:tr2} coincides  with the one defined  in \cite{MR}  by using the virtual immersion (See Appendix \ref{A:MR} for the correspondence). However, our approach as well as our description of the trace formula are different from those in \cite{MR}.

 Similar to  \cite{ACS}, the trace can be estimated using  geometry of $M\to N=G/K$. 
Hence,  we could obtain the estimate \eqref{eq:SMN} for  $\Sigma\subset M$ if there is an isometric immersion $M\to N=G/K$ satisfying that ${\rm Tr}q_{\omega}<0$ for any non-trivial $\omega\in \mathcal{H}$ on $\Sigma$ (see Proposition \ref{prop:meth1} and  \ref{prop:meth2}).
As a simple application of our results, we prove that the inequality \eqref{eq:SMN} holds for any closed minimal surface in a Berger sphere $S^{3}$ or more generally, for any closed minimal hypersurface in the geodesic hypersphere $S^{n}_{r}$ of some radius $r$ of the complex/quaternionic/octornionic projective space,  by using an isometric embedding of $S^{n}_{r}$ to the projective space  (see Theorem \ref{thm:a1} and \ref{thm:a2} for more precise results).

Moreover, when $M$ is itself a compact semi-simple RSS, by taking the trivial isometric immersion, i.e. the identity map ${id}: M\to M$, we show that  the trace of $q_{\omega}$ with respect to a certain test variation $\Phi$ is always equal to $0$ for any $\omega\in \mathcal{H}$ on any two-sided minimal hypersurface $\Sigma \subset M$ (Proposition \ref{prop:key2}. Note that the same conclusion was also obtained in \cite{MR} by using the virtual immersion). This remarkable fact is being a key of the proof of Theorem.\ref{thm:main0}.
In fact, when ${\rm Tr}q_{\omega}=0$ for any $\omega\in \mathcal{H}$,  we obtain an affine bound ${\rm index}(\Sigma)\geq C(b_{1}(\Sigma)-D)$ for some constant $C, D$ depending only on  $M$ (see Lemma \ref{lem:key} and Proposition \ref{prop:key4}), and if furthermore $\Sigma$ is unstable (i.e. ${\rm index}(\Sigma)\geq 1$),  the affine bound implies the required linear bound ${\rm index}(\Sigma)\geq C'b_{1}(\Sigma)$ for some constant $C'$ determined by $C$ and $D$. We remark that this argument has been used in \cite{ACS2} in the case when $N$ is a flat torus and in \cite{MR} by using the virtual immersion,  however, in both papers, they supposed an additional assumption that there exists a point $p\in \Sigma$ such that the principal curvatures at $p$ are distinct. In this paper, we shall improve their argument by using a non-trivial extension of some formulas, and  prove the inequality (2)  without the assumption of principal curvatures. Actually, our argument can be applied to the case when $N$ is a flat torus as well, and the result given in \cite[Theorem 1]{ACS2} is generalized to any closed minimal hypersurface in the flat torus.

The paper is organized as follows. In Section \ref{sec:ind}, we summarize our method of index estimate in a general setting based on the argument given in \cite{ACS, MR}. In Section \ref{sec:tr1}, by considering an isometric immersion of $M$ into a compact semi-simple RSS $N$, we derive the first trace formula as an extension of the formula given in \cite[Proposition 1]{ACS}. By using this, we prove that the inequality \eqref{eq:SMN} holds for any closed minimal surface embedded in a family of Berger spheres. In Section \ref{sec:tr2}, extending the second method given in \cite{ACS}, we prove another trace formula via the isometric immersion $M\to N$. Finally, we prove Theorem \ref{thm:main0}  in Section \ref{sec:pf}  by improving the argument  given in \cite{ACS2, MR}. 

\section{Index and the method of estimate}\label{sec:ind}

\subsection{Second variation of minimal hypersurfaces}
Let $M^{n}$ be an $n$-dimensional Riemannian manifold and  $\Sigma$ a closed $(n-1)$-dimensional manifold. We consider an isometric immersion $f: \Sigma\to M$. We identify the tangent space $T_{p}\Sigma$ with a subspace of $T_{f(p)}M$ via the differential map $df_{p}: T_{p}\Sigma\to T_{f(p)}M$.  We denote the orthogonal complement of  $T_{p}\Sigma$ in $T_{f(p)}M$ by $\nu_{p}\Sigma$, and we call $\nu_{p}\Sigma$ a normal space of $f$ at $p$. In our setting, $\nu_{p}\Sigma$ is a $1$-dimensional subspace. We denote the normal bundle of $f$ by $\nu\Sigma$.

We denote the Levi-Civita connection of $M$ by $\nabla^{M}$. The {\it second fundamental form} $A$ of $f$ is defined by $A(X,Y):=(\nabla^{M}_{X}Y)^{\perp}$ for $X,Y\in \Gamma(T\Sigma)$, where $\perp$ means the orthogonal projection to the normal space. The second fundamental form is a symmetric tensor field on $\Sigma$. The mean curvature vector $H$ of $f$ is defined by $H:=\sum_{\mu=1}^{n-1}A(e_{\mu},e_{\mu})$, where $\{e_{\mu}\}_{\mu=1}^{n-1}$ is an orthonormal basis of $T_{p}\Sigma$. We say $f$ is  {\it minimal} if $f$ is a critical point of the volume functional, or equivalently, 
the mean curvature vector $H$ is vanishing along $f$. In the following, we always assume $f$ is a minimal immersion of $(n-1)$-dimensional closed manifold $\Sigma$. When $f$ is an embedding, we call $\Sigma$ a {\it minimal hypersurface} in $M$.

The second variation of the volume functional for a normal variation $\{f_{s}\}_{s\in (-\epsilon,\epsilon)}$ of a minimal immersion $f=f_{0}$ is given by (cf. \cite{Simons})
\begin{align*}
\frac{d^{2}}{ds^{2}}{\rm Vol}(f_{s})\Big{|}_{s=0}=\int_{\Sigma} |\nabla^{\perp}V|^{2}-{\rm Ric}^{M}(V,V)-|A|^{2}|V|^{2}\,d\mu_{\Sigma},
\end{align*}
where  $V:=\frac{df_{s}}{ds}|_{s=0}\in \Gamma(\nu\Sigma)$, $\nabla^{\perp}$ denotes the normal connection and ${\rm Ric}^{M}$ is the Ricci tensor of $M$. We define a quadratic form on $\Gamma(\nu\Sigma)$ by
\begin{align*}
Q(V,W)=\int_{\Sigma} \langle \nabla^{\perp}V,  \nabla^{\perp}W\rangle -{\rm Ric}^{M}(V,W)-|A|^{2}\langle V,W\rangle\,d\mu_{\Sigma}\quad \textup{for $V,W\in \Gamma(\nu\Sigma)$},
\end{align*}
and $Q$ is called the {\it index form} of $f$.
A minimal immersion $f$ is said to be {\it stable} if $Q(V,V)\geq 0$ for any normal vector field $V$, otherwise we say {\it unstable}.

\begin{remark}
{\rm 
As mentioned in \cite{ACS}, we can deal with a non-orientable hypersurface. If this is the case, the symbol $d\mu_{\Sigma}$ denotes the {\it Riemannian density} of $\Sigma$ (See \cite{Lee} for details). Since the divergence theorem holds even for Riemannian densities, the computation presented in this paper works just as well as in the orientable case.
}
\end{remark}

By using the divergence theorem, we may write
\[
Q(V,W)=-\int_{\Sigma} \langle \Delta^{\perp}V+R^{M}(V)^{\perp}+|A|^{2}V,W\rangle d\mu_{\Sigma},
\]
where $\Delta^{\perp}$ is the Laplacian of $\nabla^{\perp}$ and $R^{M}(V):={\rm Tr}R^{M}(V,\cdot)\cdot$, where $R^{M}$ is the curvature tensor of $M$ defined by $R^{M}(X,Y)Z=\nabla^{M}_{X}\nabla^{M}_{Y}Z-\nabla^{M}_{Y}\nabla^{M}_{X}Z-\nabla^{M}_{[X,Y]}Z$. We put 
\[
\mathcal{J}(V):=\Delta^{\perp}V+R^{M}(V)^{\perp}+|A|^{2}V.
\]
It is known that $\mathcal{J}$ is a self-adjoint elliptic linear operator acting on $\Gamma(\nu \Sigma)$ and it has discrete eigenvalues $\lambda_{1}\leq \lambda_{2}\leq \cdots \leq \lambda_{k}\leq \cdots \to \infty$ (Note that the eigenvalue is defined by $\mathcal{J}(V)=-\lambda V$).  Then,  the {\it index} of $f$ is given by the dimension of the eigenspaces of negative eigenvalues. Also, the {\it nullity}  is defined by the dimension of $0$-eigenspaces.  We denote the index (resp. nullity) of $f: \Sigma\to M$ by ${\rm index}(\Sigma)$ (resp. ${\rm nullity}(\Sigma)$). 

An embedded hypersurface $\Sigma$ is called {\it two-sided} if the normal bundle $\nu \Sigma$ is trivial, i.e. there exists a smooth unit normal vector field $\nu$ along $\Sigma$. Otherwise, we say $\Sigma$  {\it one-sided}. If the ambient space $M$ is oriented, then $\Sigma$ is two-sided if and only if $\Sigma$ is oriented.  In general, if the normal bundle $\nu\Sigma$ of a minimal immersion $f: \Sigma^{n-1}\to M^{n}$ is trivial, then a normal vector field $V$ is written as $V=\varphi\cdot \nu$ for some $\varphi\in C^{\infty}(\Sigma)$. Thus, the second variation becomes 
\begin{align*}
&Q(\varphi,\psi)=-\int_{\Sigma} \langle \mathcal{J}(\varphi),\psi\rangle\,d\mu_{\Sigma},\quad \varphi,\psi\in C^{\infty}(\Sigma),\\
{\rm where}\quad & \mathcal{J}(\varphi):=\Delta\varphi+{\rm Ric}^{M}(\nu,\nu)\varphi+|A|^{2}\varphi.
\end{align*}
 If this is the case, the index (resp. nullity) of $f$ coincides with the dimension of eigenspaces of negative eigenvalues (resp. $0$-eigenvalue) of $\mathcal{J}$ acting on $C^{\infty}(\Sigma)$.

\subsection{Method of index estimate}

The following lemma is a generalization of argument used in \cite{ACS,MR}.

\begin{lemma}\label{lem:key}
Suppose $f: \Sigma^{n-1}\to M^{n}$ is a minimal immersion. Let  $\mathcal{G}$ be a finite dimensional vector space equipped with a positive-definite inner product $\langle,\rangle_{\mathcal{G}}$ and suppose that there is a bilinear map
\[
\Phi: \Gamma(T\Sigma)\times\mathcal{G}\to \Gamma(\nu\Sigma)
\]
so that $\Phi$ defines a quadratic form $q_{T}$ on $\mathcal{G}$ for each $T\in \Gamma(T\Sigma)$;
\[
q_{T}(v,w):=Q(\Phi_{T,v},\Phi_{T,w})=-\int_{\Sigma}\langle \mathcal{J}(\Phi_{T,v}), \Phi_{T,w}\rangle\,d\mu_{\Sigma},
\]
where $\Phi_{T,v}=\Phi(T,v)$. We also define another quadratic form  $L_{T}$ on $\mathcal{G}$  by 
\[
L_T(v,w):=\langle \Phi_{T,v},\Phi_{T,w}\rangle_{L^2(\Sigma)}=\int_{\Sigma}\langle\Phi_{T,v},\Phi_{T,w}\rangle\,d\mu_{\Sigma}.
\]
Then we have
\begin{enumerate}
\item If there exists a finite dimensional linear subspace $S\subset \Gamma(T\Sigma)$ and a constant $c\in \R$ such that 
$
{\rm Tr}_{\mathcal{G}}q_{T}<c\cdot {\rm Tr}_{\mathcal{G}}L_{T}
$
for any non-trivial element $T\in S$,
then
\[
\sharp\{\textup{eigenvalues of $\mathcal{J}$ less than $c$}\}\geq \frac{{\rm dim}S}{{\rm dim}\mathcal{G}}.
\]
\item If there exists a finite dimensional linear subspace $S\subset \Gamma(T\Sigma)$ such that
$
{\rm Tr}_{\mathcal{G}}q_{T}\leq 0
$
for any non-trivial element $T\in S$,
 then 
\[
{\rm index}(\Sigma)\geq \frac{1}{{\rm dim}\mathcal{G}}({\rm dim}S-{\rm dim}S_{0}),
\]
where 
\[
S_{0}:=\{T\in S\mid \mathcal{J}(\Phi_{T,v})\equiv0\ \forall v\in \mathcal{G}\}.
\]

\end{enumerate}

\end{lemma}

\begin{proof}
We denote the eigenvalue of $\mathcal{J}$ by  $\lambda_1\leq \lambda_2\leq \cdots \leq \lambda_{r}\leq \cdots\to \infty$.  Fix a constant $c\in \R$ and consider the set of eigenvalues $\{\lambda_{k}\}_{k=1}^{l}$ satisfying that $\lambda_{k}<c$. Then, $l=\sharp\{\textup{eigenvalues of $\mathcal{J}$ less than $c$}\}$. We take an $L^{2}$-basis $\{V_k\}_{k=1}^l$ of the eigenspaces of $\{\lambda_{k}\}_{k=1}^{l}$, and  we fix an orthonormal basis $\{\theta_{i}\}_{i=1}^{\delta}$ of $\mathcal{G}$, where $\delta={\rm dim} \mathcal{G}$.

(i) We define a linear map $\chi: S\to \R^{\delta\times l}$ by
\[
\chi(T):=\Big[\int_{\Sigma}\langle \Phi_{T,\theta_{i}},V_{k}\rangle\,d\mu_{\Sigma}\Big]_{i=1,\ldots,\delta, k=1\ldots, l},
\]
where the right hand side is regarded as a (column) vector in $\R^{\delta\times l}$.

Assume that ${\rm dim} S>\delta\times l$. Then there exists a non-trivial element $T_{0}\in S$ such that $\chi(T_{0})=0$. Namely,  $\Phi_{T_{0},\theta_{i}}$ is $L^{2}$-orthogonal to $V_{k}$ for any $i=1,\ldots,\delta$ and $k=1,\ldots ,l$, and hence, we obtain
\begin{align*}
Q(\Phi_{T_{0},\theta_{i}},\Phi_{T_{0},\theta_{i}})\geq \lambda_{l+1}\int_{\Sigma}|\Phi_{T_{0},\theta_{i}}|^{2}\,d\mu_{\Sigma}.
\end{align*}
Summing up for all $i=1,\ldots,\delta$, we have 
\begin{align*}
{\rm Tr}_{\mathcal{G}}q_{T_{0}}\geq \lambda_{l+1}{\rm Tr}_{\mathcal{G}}L_{T_0}.
\end{align*}
Note that we used the assumption that $\langle,\rangle_{\mathcal{G}}$ is positive-definite, i.e. the trace of a quadratic form $q$ is defined by ${\rm Tr}_{\mathcal{G}}q=\sum_{i=1}^{\delta}q(\theta_{i},\theta_{i})$.

However,  since we assume ${\rm Tr}_{\mathcal{G}}q_{T_{0}}<c\cdot {\rm Tr}_{\mathcal{G}}L_{T_0}\leq   \lambda_{l+1}{\rm Tr}_{\mathcal{G}}L_{T_0}$, this is a contradiction. Therefore, we obtain ${\rm dim} S\leq \delta\times l$ and this proves (i).

(ii)  We slightly modify the proof of (i). Let $S_{0}^{\perp}$ be the orthogonal complement of $S_{0}$ in $S\subset \Gamma(T\Sigma)$. Take an $L^{2}$-basis $\{V_{k}\}_{k=1}^{l}$ of {\it negative} eigenfunctions of $\mathcal{J}$ (namely we set $c=0$ in the proof of (i)) and let  $\chi: S\to \R^{\delta\times l}$ be the same map given in the proof of (i).
We consider the restricted map  $\chi'=\chi|_{S_{0}^{\perp}}: S_{0}^{\perp}\to {\R}^{\delta\times l}$. Assume that ${\rm dim}S_{0}^{\perp}>\delta\times l$. Then there exists a non-trivial element $T_{0}\in S_{0}^{\perp}$ so that $\chi'(T_{0})=0$. Then $\Phi_{T_{0},\theta_{i}}$ is $L^{2}$-orthogonal to any negative eigenvector, and hence, we obtain 
\begin{align}\label{eq:posi}
Q(\Phi_{T_{0},\theta_{i}},\Phi_{T_{0},\theta_{i}})\geq 0
\end{align}
for any $i$.
In particular, we have ${\rm Tr}_{\mathcal{G}}q_{T_{0}}\geq 0$ (as $\langle, \rangle_{\mathcal{G}}$ is positive-definite).  Since we assume ${\rm Tr}_{\mathcal{G}}q_{T_{0}}\leq 0$, we obtain ${\rm Tr}_{\mathcal{G}}q_{T_{0}}=0$.  By \eqref{eq:posi},
this implies 
\begin{align}\label{eq:nul}
Q(\Phi_{T_{0},\theta_{i}},\Phi_{T_{0},\theta_{i}})=-\int_{\Sigma}\langle \mathcal{J}(\Phi_{T_{0},\theta_{i}}),\Phi_{T_{0},\theta_{i}}\rangle=0
\end{align}
for any $i$. Because $\Phi_{T_{0},\theta_{i}}$ is $L^{2}$-orthogonal to any negative eigenvector, \eqref{eq:nul} implies $\mathcal{J}(\Phi_{T_{0},\theta_{i}})=0$ for any $i$, i.e. $T_{0}\in S_{0}$. This contradicts the facts that $T_{0}\in S_{0}^{\perp}$ and $T_{0}\neq 0$. Therefore, ${\rm dim}S_{0}^{\perp}\leq \delta\times l$. This proves (ii).
\end{proof}

As a special situation, we consider the set of harmonic $1$-forms  $\mathcal{H}$ on $\Sigma$ and denote  its metric dual by $\mathcal{H}^{*}\subset \Gamma(T\Sigma)$. Note that, by Hodge theory, we have ${\rm dim}\mathcal{H}^{*}=b_{1}(\Sigma)$, where $b_{1}(\Sigma)$ is the first Betti number of $\Sigma$.  By letting $S=\mathcal{H}^{*}$ in the previous lemma, we obtain the following result as an extension of \cite[Theorem A]{ACS} and \cite[Proposition 10]{MR}:

\begin{proposition}\label{prop:ind}
Let $f: \Sigma^{n-1}\to M^{n}$ be a minimal immersion. Suppose the same assumption given in Lemma {\rm \ref{lem:key}}.
\begin{enumerate}
\item 
 If ${\rm Tr}_{\mathcal{G}}q_{\omega^{\sharp}}<0$ (strictly negative) for any non-trivial $\omega\in \mathcal{H}$, then it holds that 
\begin{align*}
{\rm index}(\Sigma)\geq \frac{1}{{\rm dim}\mathcal{G}}b_{1}(\Sigma).
\end{align*}
\item 
If ${\rm Tr}_{\mathcal{G}}q_{\omega^\sharp}\leq 0$ (non-positive) for any non-trivial  $\omega\in \mathcal{H}$, then it holds that 
\begin{align*}
&{\rm index}(\Sigma)+{\rm nullity}(\Sigma)\geq  \frac{1}{{\rm dim}\mathcal{G}}b_{1}(\Sigma)\quad \textup{and}\\
&{\rm index}(\Sigma)\geq \frac{1}{{\rm dim}\mathcal{G}}(b_{1}(\Sigma)-{\rm dim}\mathcal{H}_{0}),
\end{align*}
where    $\mathcal{H}_{0}\subseteq \mathcal{H}$ is a linear subspace defined by
\[
\mathcal{H}_{0}:=\{\omega\in \mathcal{H}\mid \mathcal{J}(\Phi_{\omega^{\sharp},v})\equiv0\ \forall v\in \mathcal{G}\}.
\]
\end{enumerate}
\end{proposition}

The proof of our main result (Theorem \ref{thm:main}) will be done  by taking an appropriate map $\Phi: \Gamma(T\Sigma)\times\mathcal{G}\to \Gamma(\nu\Sigma)$ and applying Proposition \ref{prop:ind} to the map $\Phi$.  In order to construct $\Phi$, we consider an isometric immersion $M^{n}$ into an $(n+k)$-dimensional Riemannian symmetric space $N^{n+k}$. When $N=\R^{n+k}$, the method has been considered by  Ambrozio-Carlotto-Sharp \cite{ACS, ACS2} which is inspired by the idea of  Ros \cite{Ros}, Urbano \cite{Urbano} and Savo \cite{Savo}. In the present paper, we extend their method to the case when $N$ is a compact semi-simple Riemannian symmetric space (see also Appendix for the correspondence of our argument in the case when $N=\R^{n+k}$). We shall give the details of our method in Section \ref{sec:tr1} and \ref{sec:tr2}.

\subsection{Preliminaries on symmetric spaces}
A Riemannian manifold  $(N,g)$ is called a {\it Riemannian symmetric space} ({\it RSS} for short) if,  for each $p\in N$, there is an isometry $s_{p}$ on $N$ satisfying that (i) $s_{p}^{2}=id_{N}$ and (ii) $p$ is an isolated fixed point of $s_{p}$. Note that $(N,g)$ is a complete Riemannian manifold.  Moreover, it is known that the identity component of the isometry group $I_{0}(N)$ acts on $N$ transitively and thus, $N$ is diffeomorphic to the homogeneous space $G/K_{p}$, where $G=I_{0}(N)$ and $K_{p}$ is the isotropy subgroup at some point $p\in N$.  In the following, we identify $N$ with $G/K_{p}$. We refer to \cite{Sakai} for basic facts for RSS. 

We denote the Lie algebra of $G$ (resp. $K_{p}$) by $\fg$ (resp. $\fk_{p}$).   The geodesic symmetry $s_{p}$ at $p$ defines an involutive automorphism $\widetilde{\sigma}_{p}$ on $G=I_{0}(N)$ by $\widetilde{\sigma}_{p}(g):=s_{p}gs_{p}$. Then, the differential $\sigma_{p}:=(d\widetilde{\sigma}_{p})_{e}$ at the identity element $e\in G$ induces an involutive automorpshim $\sigma_{p}:\fg\to \fg$, and the eigen-decomposition of $\sigma_{p}$ yields a reductive decomposition 
\[
\fg=\fk_{p}\oplus \fn_{p},
\]
where $\fk_{p}$ coincides with the $+1$-eigenspace of $\sigma_{p}$ and $\fn_{p}$ denotes the $-1$-eigenspace of $\sigma_{p}$.  
The  decomposition $\fg=\fk_{p}\oplus \fn_{p}$  satisfies the following bracket relations
\begin{align}\label{eq:br0}
[\fk_{p},\fk_{p}]\subset \fk_{p},\quad  [\fk_{p},\fn_{p}]\subset \fn_{p}\quad [\fn_{p},\fn_{p}]\subset \fk_{p}.
\end{align}

For an element $X\in \fg$, we define the {\it fundamental vector field} $X^{\dagger}$ on $N=G/K_{p}$ by 
\[
X^{\dagger}(q):=\frac{d}{dt} {\rm exp}(tX)q\Big{|}_{t=0},\quad q\in N.
\]
Then, $X^{\dagger}$ defines a Killing vector field on $N$ since ${\rm exp}(tX)\in G=I_{0}(N)$ is a 1-parameter subgroup of isometries. It is a general fact for Riemannian homogeneous spaces that the map $X\mapsto X^{\dagger}$ is a liner map satisfying that 
\begin{align}\label{eq:br1}
[X^{\dagger}, Y^{\dagger}]=-[X,Y]^{\dagger}.
\end{align}
Moreover, for each $p\in N$, the linear map
\[
\fn_{p}\to T_{p}N,\quad X\mapsto X^{\dagger}(p)
\]
yields an isomorphism. We denote the converse correspondence by $T_{p}N\ni V\mapsto \widetilde{V}\in \fn_{p}$.

The following lemma is a fundamental fact on symmetric space, but will be a key  in our argument. 

\begin{lemma}\label{lem:key1}
Suppose $N$ is a Riemannian symmetric space and fix arbitrary point $p\in N$.  We denote the  Levi-Civita connection of $N$ by $\overline{\nabla}$.
  \begin{enumerate}
\item If $X\in \fk_{p}$ then $X^{\dagger}(p)=0$, and $\onab_{V}X^{\dagger}(p)$ corresponds to $[X, \widetilde{V}]\in \fn_{p}$ for any $V\in T_{p}N$.
\item If $X\in \fn_{p}$ then $\overline{\nabla}X^{\dagger}(p)=0$.
\end{enumerate}
\end{lemma}
\begin{proof}
We give a proof for the sake of convenience of the reader. 

(i) If $X\in \fk_{p}$, then we have ${\rm exp}(tX)\in K_{p}$. Thus, $X^{\dagger}(p)=\frac{d}{dt}{\rm exp}(tX)p|_{t=0}=\frac{d}{dt}p|_{t=0}=0$  since $K_{p}$ is the isotropy subgroup at $p$. Moreover, since $V_{p}=\widetilde{V}^{\dagger}_{p}$, we have
\begin{align*}
\onab_{V}X^{\dagger}(p)=\onab_{\widetilde{V}^{\dagger}}X^{\dagger}(p)=[\widetilde{V}^{\dagger}, {X}^{\dagger}](p)=[X,\widetilde{V}]^{\dagger}(p)
\end{align*}
where we used $X^{\dagger}(p)=0$ and \eqref{eq:br1}. Since $[X,\widetilde{V}]\in \fn_{p}$ by \eqref{eq:br0}, this shows that $\onab_{V}X^{\dagger}(p)$ corresponds to $[X,\widetilde{V}]$ via the canonical identification $T_{p}N\simeq \fn_{p}$.

(ii) By (i) and the bracket relation \eqref{eq:br0}, we have
\begin{align}\label{eq:br2}
[Y,Z]^{\dagger}(p)=0\quad \textup{for any $Y,Z\in \fn_{p}$}
\end{align}
since $[Y,Z]\in \fk_{p}$.
On the other hand, since $X^{\dagger}$ is a Killing vector field on $N$, we have
\begin{align*}
X^{\dagger}\langle y,z\rangle&=\langle [X^{\dagger},y],z\rangle+\langle y,[X^{\dagger},z]\rangle
\end{align*}
for any vector fields $y, z$ on $N$.  Thus,  the Koszul formula and \eqref{eq:br1} show that
\begin{align*}
2\langle \overline{\nabla}_{X^{\dagger}}Y^{\dagger},Z^{\dagger}\rangle&=X^{\dagger}\langle Y^{\dagger},Z^{\dagger}\rangle+Y^{\dagger}\langle X^{\dagger},Z^{\dagger}\rangle-Z^{\dagger}\langle X^{\dagger},Y^{\dagger}\rangle\\
&\quad +\langle [X^{\dagger}, Y^{\dagger}],Z^{\dagger}\rangle-\langle [X^{\dagger},Z^{\dagger}], Y^{\dagger}\rangle-\langle [Y^{\dagger},Z^{\dagger}],X^{\dagger}\rangle\\
&=\langle [X^{\dagger}, Y^{\dagger}],Z^{\dagger}\rangle+\langle[X^{\dagger},Z^{\dagger}], Y^{\dagger}\rangle+\langle[Y^{\dagger},Z^{\dagger}], X^{\dagger}\rangle\\
&=-\langle [X, Y]^{\dagger},Z^{\dagger}\rangle-\langle[X,Z]^{\dagger}, Y^{\dagger}\rangle-\langle[Y,Z]^{\dagger}, X^{\dagger}\rangle
\end{align*}
for any $X,Y,Z\in \fg$.  In particular, by \eqref{eq:br2}, we obtain $\langle \overline{\nabla}_{X^{\dagger}}Y^{\dagger},Z^{\dagger}\rangle(p)=0$ for any $X,Y,Z\in \fn_{p}$. Since $\fn_{p}$ is isomorphic to $T_{p}N$ via the map $X\mapsto X^{\dagger}(p)$, this shows that $ \overline{\nabla}Y^{\dagger}(p)=0$. Because $Y\in \fn_{p}$ is arbitrary, this proves the second assertion. 
\end{proof}

A Riemannian symmetric space $N$ is said to be {\it semi-simple} if the identity component of the isometry group $G=I_{0}(N)$ is a semi-simple Lie group. In the following, we assume $N=G/K_{p}$ is a semi-simple RSS.  Note that this assumption is not restrictive. In fact, it is known that any simply-connected RSS is decomposed into a product of Euclidean space and a semi-simple RSS. Moreover, any semi-simple RSS is an Einstein manifold, and if furthermore $G$ is compact, then $N$ has positive Ricci curvature.  We remark that  the Euclidean space $\R^{n+k}$ is not a semi-simple RSS and somehow a special case. We will discuss the case when $N=\R^{n+k}$ in Appendix.  

When $G$ is semi-simple, then there is a canonical  bilinear form ${\bf B}$ on the Lie algebra $\fg$ of $G$, called the  {\it Killing  form}.  It is known that ${\bf B}$ is non-degenerate and invariant under any Lie algebra automorphism on $\fg$. 
Moreover, if $G$ is compact, then ${\bf B}$ is negative-definite. Thus, if this is the case, we can define a positive-definite inner product on $\fg$ by
\[
\langle X,Y\rangle_{\fg}:=-{\bf B}(X,Y).
\]
Actually, this inner product induces the canonical $G$-invariant Riemannian metric on a compact semi-simple symmetric space $N=G/K_{p}$. More precisely, the restriction of $\langle,\rangle_{\fg}$ to the subspace $\fn_{p}$ defines an ${\rm Ad}(K_{p})$-invariant inner product $\langle,\rangle_{\fn_{p}}$ on $\fn_{p}$. Then,  $\langle,\rangle_{\fn_{p}}$ extends to the $G$-invariant Riemannian metric on $N=G/K_{p}$ by the left $G$-action. In particular, $(\fn_{p},\langle,\rangle_{\fn_{p}})$ and $(T_{p}N, g_{p})$ are isometric  by the canonical isomorphism $\fn_{p}\simeq T_{p}N$.

We use the following simple fact for the canonical inner product:

\begin{lemma}\label{lem:key2}
Suppose $G$ is compact and  semi-simple. Then, the canonical decomposition $\fg=\fk_{p}\oplus \fn_{p}$ is an orthogonal decomposition w.r.t. $\langle,\rangle_{\fg}$ for every $p\in N$.
\end{lemma}
\begin{proof}
 Take arbitrary $X\in \fk_{p}$ and $Y\in \fn_{p}$. Since the Killing form is invariant for any automorphism on the Lie algebra, we have
\begin{align*}
\langle X, Y\rangle_{\fg}=-{\bf B}(X,Y)=-{\bf B}(d\sigma_{p}X,d\sigma_{p}Y)={\bf B}(X,Y)=-\langle X,Y\rangle_{\fg}
\end{align*}
and hence, $\langle X, Y\rangle _{\fg}=0$.
\end{proof}

\begin{remark}
{\rm
When $G$ is non-compact, one can define a positive-definite inner product on $\fg$ by
$
\langle X, Y\rangle_{\fg}':=-{\bf B}(X,\sigma_{o}Y) 
$, where $\sigma_{o}: \fg\to \fg$ is the Cartan involution at a fixed point $o\in N$. Actually, $\langle , \rangle_{\fg}'$ induces the canonical $G$-invariant Riemannian metric on $N$. However,  we remark that $\langle,\rangle_{\fg}'$ depends on the fixed point $o\in N$ and because of this,  the canonical  decomposition $\fg=\fk_{p}\oplus \fn_{p}$ at $p\in N$ may not orthogonal w.r.t. $\langle,\rangle_{\fg}'$,  although the decomposition $\fg=\fk_{o}\oplus \fn_{o}$ at $o$ is  orthogonal. Lemma \ref{lem:key2} is crucial in our argument, and our computations given in Section \ref{sec:tr1} and \ref{sec:tr2} do not work if we choose $\langle , \rangle_{\fg}'$ as the inner product on the Lie algebra $\fg$ of a non-compact Lie group $G$.
}
\end{remark}

\section{Trace formula I}\label{sec:tr1}

Let $(M^{n},g)$ be a Riemannian manifold and $f: \Sigma^{n-1}\to M^{n}$ an isometric immersion of closed manifold $\Sigma^{n-1}$. Throughout this section, we suppose that  {\it $f$ has trivial normal bundle}  (i.e   $\Sigma$ is  {\it two-sided}  if  $f$ is an embedding) so that there is a smooth unit normal vector field $\nu$ along $\Sigma$. In this section, we provide a trace formula for a quadratic form by extending the method considered in \cite[Proposition 1]{ACS}, and give a simple application. 

\subsection{Trace formula}

Suppose that $M^{n}$ is isometrically immersed into a  compact semi-simple symmetric space $N^{n+k}=G/K$ via the map $F: M\to N$. Then, $\Sigma$ is isometrically immersed into $N$ by the map $F\circ f: \Sigma\to N$. We consider the following map:
\[
\Psi: \Gamma(T\Sigma)\times \fg\to \Gamma(\nu\Sigma),\quad \Psi(T,X)=\psi_{T,X}\nu,
\]
where
\[
\psi_{T,X}(p):=\langle T, X^{\dagger}\rangle_{F\circ f(p)}. 
\]
Also, we consider  a quadratic form on $\fg$ defined by
\[
r_{T}(X,Y):=Q(\psi_{T,X}, \psi_{T,Y}).
\]
We shall compute the trace of $r_{T}$ over $\fg$ with respect to  the inner product $\langle,\rangle_{\fg}$ defined by the negative Killing form. We put 
\[
r_{T,p}(X,Y):=\langle \mathcal{J}(\psi_{T, X}), \psi_{T,Y}\rangle(p)
\]
for $p\in \Sigma$
so that 
\[
{\rm Tr}_{\fg}r_{T}=-\int_{\Sigma}{\rm Tr}_{\fg}r_{T, p}\,d\mu_{\Sigma}.
\]
In the following, we abbreviate the fixed subscript $T$.  

We use the following notations:
\begin{itemize}
\item $\overline{\nabla}$, $\nabla^{M}$, $\nabla$: the Levi-Civita connections of $N$, $M$ and $\Sigma$, respectively.
\item $\overline{R}, R^{M}, R$, the curvature tensors of $\overline{\nabla}$, $\nabla^{M}$ and $\nabla$, respectively.  Note that we use the following definition of the curvature tensor (the sign is different from \cite{ACS}):
\[
\overline{R}(X,Y)Z:=\onab_{X}\onab_{Y}Z-\onab_{Y}\onab_{X}Z-\onab_{[X,Y]}Z.
\]
\item $\overline{\rm Ric}$, ${\rm Ric}^{M}$: the Ricci  tensor of $\overline{\nabla}$ and $\nabla^{M}$, respectively. 
\item $A,B$: the second fundamental forms of $\Sigma\to M$ and $M\to N$, respectively.
\item For a fixed point $p\in\Sigma$, we take a geodesic normal coordinate around $p$ in $\Sigma$, and we denote the coordinate basis by $\{\p_{\mu}\}_{\mu=1}^{n-1}$.
\end{itemize}

In this notation, we have the following fundamental formulas.
\begin{itemize}
\item (Gauss equation) For any $X,Y,Z,W\in T_p\Sigma$, we have 
\[\langle R^M(X,Y)Z,W\rangle=\langle R(X,Y)Z,W\rangle+\langle A(X,Z), A(Y,W)\rangle-\langle A(X,W), A(Y,Z)\rangle.\]
Gauss equation for the isometric immersion $M\to N$ is given as well. 
\item (Codazzi equation) For any $X,Y,Z\in T_p\Sigma$, we have
\[
(R^M(X,Y)Z)^\perp=(\nabla^\perp_XA)(Y,Z)-(\nabla_Y^\perp A)(X,Z),
\]
where $\nabla^\perp$ is the normal connection defined by $\nabla_X^\perp \xi:=(\nabla^M_X\xi)^\perp$ for a normal vector field $\xi$ along $\Sigma$ in $M$ and $(\nabla^\perp_XA)(Y,Z)=\nabla_X^\perp(A(Y,Z))-A(\nabla_XY,Z)-A(Y,\nabla_XZ)$.
\end{itemize}

We first prove the following general formula.

\begin{lemma}
Let $M^{n}$ be a Riemannian manifold and $\Sigma^{n-1}\to M^{n}$  a minimal immersion of a closed manifold $\Sigma^{n-1}$. For any tangent vector field $T$ on $\Sigma$, we have
\begin{align}\label{eq:BW}
\frac{1}{2}\Delta |T|^{2}-|\nabla T|^{2}=-(\Delta_{1}T^{\flat})(T)+\sum_{\mu=1}^{n-1}\langle {R}^{M}(T,\p_{\mu})\p_{\mu}, T\rangle-\sum_{\mu=1}^{n-1}| A(\p_{\mu},T)|^2,
\end{align}
where $T^{\flat}=\langle T,\cdot\rangle|_{\Sigma}$  and $\Delta_{1}$ is the Hodge Laplacian action on $\Omega^{1}(\Sigma)$  the space of $1$-forms on $\Sigma$.
\end{lemma}

\begin{proof}
Fix arbitrary point $p\in \Sigma$ and we take a geodesic normal coordinate in $\Sigma$. 
By the Bochner-Weitzenb\"ock formula and the Gauss equation for minimal immersion $\Sigma\to M$, we have
\begin{align*}
\frac{1}{2}\Delta |T|^{2}-|\nabla T|^{2}
&=\sum_{\mu=1}^{n-1}\langle \nabla_{\p_{\mu}}\nabla_{\p_{\mu}}T,T\rangle(p)\\
&=-(\Delta_{1}T^{\flat})(T)+\sum_{\mu=1}^{n-1}\langle R(T,\p_{\mu})\p_{\mu}, T\rangle\nonumber\\
&=-(\Delta_{1}T^{\flat})(T)+\sum_{\mu=1}^{n-1}\langle {R}^{M}(T,\p_{\mu})\p_{\mu}, T\rangle-\sum_{\mu=1}^{n-1}| A(\p_{\mu},T)|^2.\nonumber
\end{align*}
This proves the formula.
\end{proof}

Next, we compute the trace of $r_p$ for a fixed point $p\in \Sigma$.

\begin{lemma}
\begin{align}\label{eq:trr}
{\rm Tr}_{\fg}r_{p}&=-(\Delta_{1}T^{\flat})(T)-2\sum_{\mu=1}^{n-1}|A(\p_{\mu},T)|^{2}+|A|^{2}|T|^{2}+{\rm Ric}^{M}(\nu,\nu)|T|^{2}\\
&\quad-\sum_{\mu=1}^{n-1}|B(\p_{\mu},T)|^{2} +\sum_{\mu=1}^{n-1}\langle {R}^{M}(T,\p_{\mu})\p_{\mu}, T\rangle-\sum_{\mu=1}^{n-1}\langle \overline{R}(T,\p_{\mu})\p_{\mu}, T\rangle.\nonumber
\end{align}
\end{lemma}
\begin{proof}
Since the trace is independent of the choice of orthonormal basis, we shall take a preferred orthonormal basis of $\fg$. We put  $d={\rm dim}\fg$.
By Lemma \ref{lem:key2}, the canonical decomposition $\fg=\fk_{p}\oplus\ \fn_{p}$ is orthogonal, and hence, we can take an orthonormal basis $\{X_{1},\ldots, X_{d-(n+k)}, Y_{1},\ldots, Y_{n+k}\}$ of $\fg$ so that $\{X_{1},\ldots, X_{d-(n+k)}\}$ is an orthonormal basis of $\fk_{p}$ and $\{Y_{1},\ldots, Y_{n+k}\}$ is an orthonormal basis of $\fn_{p}$.  Since $X^{\dagger}(p)=0$ by Lemma \ref{lem:key1}, we see $r_{p}(X_{i}, X_{i})=0$ for any $i=1,\ldots, d-(n+k)$. Therefore, we have
\[
{\rm Tr}_{\fg}r_{p}=\sum_{i=1}^{n+k} r_{p}(Y_{i}, Y_{i})(p)=\sum_{i=1}^{n+k} \Big(\Delta\psi_{Y_{i}}\cdot \psi_{Y_{i}}(p)+({\rm Ric}^{M}(\nu,\nu)+|A|^{2})\psi_{Y_{i}}^{2}(p)\Big).
\]
Since the linear isomorphism $\fn_{p}\simeq T_{p}N$ is isometric, $\{Y_{i}^{\dagger}(p)\}_{i=1}^{n+k}$ becomes an orthonormal basis of $T_{p}N$. Therefore, we see
\[
\sum_{i=1}^{n+k}\psi_{Y_{i}}^{2}(p)=\sum_{i=1}^{n+k}\langle T, Y_{i}^{\dagger}\rangle^{2}(p)=|T|^{2}.
\]

  In the following computations, we take a geodesic normal coordinate of the hypersurface $\Sigma$ around $p$.
By Lemma \ref{lem:key1}, we have $\onab Y_{i}^{\dagger}(p)=0$, and hence, we see
\begin{align}
\Delta\psi_{Y_{i}}\cdot \psi_{Y_{i}}(p)\nonumber
&=\sum_{\mu=1}^{n-1}\p_{\mu}\p_{\mu}\langle T,Y_{i}^{\dagger}\rangle(p)\cdot \langle T, Y_{i}^{\dagger}\rangle(p)\nonumber\\
\label{eq:delta2}
&=\sum_{\mu=1}^{n-1}\langle \onab_{\p_{\mu}}\onab_{\p_{\mu}}T, Y_{i}^{\dagger}\rangle(p)\cdot \langle T, Y_{i}^{\dagger}\rangle(p)
+\sum_{\mu=1}^{n-1}\langle T, \onab_{\p_{\mu}}\onab_{\p_{\mu}}Y_{i}^{\dagger}\rangle(p)\cdot \langle T, Y_{i}^{\dagger}\rangle(p)
\end{align}
Because $\{Y_{i}^{\dagger}(p)\}_{i=1}^{n+k}$ is an orthonormal basis of $T_{p}N$,  the first term in \eqref{eq:delta2} is computed by 
\begin{align*}
\sum_{i=1}^{n+k}\sum_{\mu=1}^{n-1}\langle \onab_{\p_{\mu}}\onab_{\p_{\mu}}T, Y_{i}^{\dagger}\rangle(p)\cdot \langle T, Y_{i}^{\dagger}\rangle(p)
&=\sum_{\mu=1}^{n-1}\langle \onab_{\p_{\mu}}\onab_{\p_{\mu}}T, T\rangle\\
&=\frac{1}{2}\Delta |T|^{2}-|\nabla T|^{2}-\sum_{\mu=1}^{n-1}|A(\p_{\mu},T)|^{2}-\sum_{\mu=1}^{n-1}|B(\p_{\mu},T)|^{2},
\end{align*}
where we used the relation $\onab_{\p_{\mu}}T=\nabla_{\p_{\mu}}T+A(\p_{\mu},T)+B(\p_{\mu}, T)$. Thus, by using  \eqref{eq:BW}, the last equation becomes
\begin{align*}
-(\Delta_{1}T^{\flat})(T)+\sum_{\mu=1}^{n-1}\langle {R}^{M}(T,\p_{\mu})\p_{\mu}, T\rangle-2\sum_{\mu=1}^{n-1}|A(\p_{\mu},T)|^{2}-\sum_{\mu=1}^{n-1}|B(\p_{\mu},T)|^{2}.
\end{align*}

Next, we consider the second term of \eqref{eq:delta2}.
Since $Y^{\dagger}$ is a Killing vector field on $N$,  we have
$
(\onab^{2}Y^{\dagger})(Z,W)=-\overline{R}(Y^{\dagger},Z)W
$
 for any $Z,W\in T_{p}N$.  Combining this with $\onab Y_{i}^{\dagger}(p)=0$,  we have
 \begin{align}\label{eq:kill}
 (\onab_{\partial_{\mu}}\onab_{\partial_{\mu}}Y_{i}^{\dagger})(p)=(\onab^{2}Y_{i}^{\dagger})(\p_{\mu},\p_{\mu})(p)=-\overline{R}(Y_{i}^{\dagger}, \p_{\mu})\p_{\mu}.
\end{align}
Therefore, the second term of \eqref{eq:delta2} becomes
\begin{align*}
\sum_{i=1}^{n+k}\sum_{\mu=1}^{n-1}\langle T, \onab_{\p_{\mu}}\onab_{\p_{\mu}}Y_{i}^{\dagger}\rangle(p)\cdot \langle T, Y_{i}^{\dagger}\rangle(p)
&=\sum_{i=1}^{n+k}\sum_{\mu=1}^{n-1}-\langle T, \overline{R}(Y_{i}^{\dagger},\p_{\mu})\p_{\mu}\rangle(p)\cdot \langle T, Y_{i}^{\dagger}\rangle(p)\\
&=-\sum_{\mu=1}^{n-1}\langle \overline{R}(T,\p_{\mu})\p_{\mu}, T\rangle.
\end{align*}
Substituting these to \eqref{eq:delta2}, we obtain
\begin{align*}
\Delta\psi_{Y_{i}}\cdot \psi_{Y_{i}}(p)&=-(\Delta_{1}T^{\flat})(T)-2\sum_{\mu=1}^{n-1}|A(\p_{\mu},T)|^{2}-\sum_{\mu=1}^{n-1}|B(\p_{\mu},T)|^{2}\\
&\quad +\sum_{\mu=1}^{n-1}\langle {R}^{M}(T,\p_{\mu})\p_{\mu}, T\rangle-\sum_{\mu=1}^{n-1}\langle \overline{R}(T,\p_{\mu})\p_{\mu}, T\rangle.
\end{align*}
This implies the desired formula. 
\end{proof}

If we choose $T=\omega^{\sharp}$ the metric dual of a harmonic $1$-form $\omega$ on $\Sigma$, then we have $\Delta_{1}T^{\flat}=\Delta_{1}\omega=0$ and hence, by \eqref{eq:trr}, we obtain  
\begin{align}\label{eq:trr2}
{\rm Tr}_{\fg}r_{\omega^{\sharp}}&=\int_{\Sigma}2\sum_{\mu=1}^{n-1}|A(\p_{\mu},\omega^{\sharp})|^{2}-|A|^{2}|\omega|^{2} -{\rm Ric}^{M}(\nu,\nu)|\omega|^{2}\\
&\qquad+\sum_{\mu=1}^{n-1}\Big(|B(\p_{\mu},\omega^{\sharp})|^{2}-\langle {R}^{M}(\omega^{\sharp},\p_{\mu})\p_{\mu}, \omega^{\sharp}\rangle+\langle \overline{R}(\omega^{\sharp},\p_{\mu})\p_{\mu}, \omega^{\sharp}\rangle\Big) \,d\mu_{\Sigma}.\nonumber
\end{align}

%
%
%

If furthermore, $\Sigma$ is a surface,  we obtain the following which extends the formula given in  \cite[Proposition 1]{ACS}.

\begin{proposition}\label{prop:tr1}
Suppose $M^{3}$ is a $3$-dimensional Riemannian manifold isometrically immersed into a compact semi-simple symmetric space $N=G/K$. Let $\Sigma$ be a closed two-sided minimal surface in $M^{3}$. Then, for any harmonic $1$-form $\omega$ on $\Sigma$, we have
\begin{align}
\label{eq:trr3}
{\rm Tr}_{\fg}r_{\omega^{\sharp}}
&=\int_{\Sigma}-{\rm Ric}^{M}(\nu,\nu)|\omega|^{2}+\sum_{\mu=1}^{2}\Big(2|B(\p_{\mu},\omega^{\sharp})|^{2}-\langle  B(\p_{\mu},\p_{\mu}), B(\omega^{\sharp},\omega^{\sharp})\rangle\Big) \,d\mu_{\Sigma}.\\ 
\label{eq:trr4}
&=
\int_{\Sigma} \sum_{\mu=1}^{2} |B(e_{\mu}, \omega^{\sharp})|^{2}-\frac{{\rm scl}(M)}{2}|\omega|^{2}+\sum_{\mu=1}^{2}\langle \overline{R}(\omega^{\sharp},\p_{\mu})\p_{\mu}, \omega^{\sharp}\rangle\,d\mu_{\Sigma},
\end{align}
where ${\rm scl}(M)$ is the scalar curvature of $M$.
\end{proposition}

\begin{proof}
We may assume  $\{\p_{1},\p_{2}\}$ is an orthonormal basis of $T_{p}\Sigma$ satisfying that $A(\p_{i},\p_{j})=\lambda_{i}\delta_{ij}$, where $\lambda_{1}$ and $\lambda_{2}$ are eigenvalues of $A$.  Since $\Sigma$ is a minimal surface, we have $\lambda=\lambda_{1}=-\lambda_{2}$, and hence, we see
\[
2\sum_{\mu=1}^{2}|A(\p_{\mu},\omega^{\sharp})|^{2}-|A|^{2}|\omega|^{2}=0.
\]
Substituting this to \eqref{eq:trr2} and by using the Gauss equation for $M\to N$, we obtain \eqref{eq:trr3}.
Moreover, the Gauss equation for the minimal surface $\Sigma\to M^{3}$ shows that  (see also \cite{ACS})
\[
\sum_{\mu=1}^{n-1}\langle {R}^{M}(\omega^{\sharp},\p_{\mu})\p_{\mu}, \omega^{\sharp}\rangle+{\rm Ric}^{M}(\nu,\nu)|\omega|^{2}=\frac{{\rm scl}(M)}{2}|\omega|^{2}.
\]
Substituting this to \eqref{eq:trr2}, we obtain \eqref{eq:trr4}.
\end{proof}

By Proposition \ref{prop:ind}-(i), we have the following.
\begin{proposition}\label{prop:meth1}
Let $M^{3}$ be a $3$-dimensional Riemannian manifold and $\Sigma$ a closed two-sided minimal surface embedded in $M^{3}$. If there is an isometric immersion $F: M^{3}\to N^{3+k}$ into a compact semi-simple Riemannian symmetric space $N=G/K$ so that ${\rm Tr}_{\fg}r_{\omega^{\sharp}}<0$ for any non-trivial harmonic $1$-form $\omega$ on $\Sigma$, then we have
\[
{\rm index}(\Sigma)\geq \frac{1}{d}b_{1}(\Sigma),
\]
 where $d={\rm dim}G={\rm dim}{\rm Isom}_{0}(N)$.
\end{proposition}

\subsection{An application} We give a simple application of the trace formula \eqref{eq:trr3}. We take the complex projective space $\C P^{2}$ as a symmetric space $N$, and consider a real hypersurface $M^3$ in $\C P^2$. Note that, as a symmetric space, $\C P^{2}$ is isometric to $G/K=SU(3)/S(U(1)\times U(2))$.  Without loss of generality, we may assume that the holomorphic sectional curvature of $\C P^{2}$ is equal to $4$.

Let $S_{r}^{3}$ be the geodesic hypersphere of geodesic radius $r$ in $\C P^{2}(4)$. 
Note that $r\in (0,\pi/2)$ (where $\pi/2$ is equal to the injectivity radius of $\C P^{m}(4)$) and $S_r$ coincides with an orbit of the isotropy subgroup $S(U(1)\times U(2))$. Moreover, $S_r$ is diffeomorphic to the $3$-dimensional sphere $S^3$, and  $S_r$ equipped with the induced metric from $\C P^2(4)$ is isometric to a {\it Berger sphere}.

It is known  that $S_{r}$  is an {\it $\eta$-umbilical} hypersurface in $\C P^{2}(4)$. Namely, for a unit normal vector field $\nu_S$ of $S_{r}$ in $\C P^2(4)$, we define a $1$-form $\eta$ on $S_{r}$ by $\eta(X):=-\langle J\nu_S, X\rangle$, where $J$ is the complex structure on $\C P^{2}$, and then the second fundamental form of $S_{r}$ is given by
\begin{align}\label{eq:umb}
B(X,Y)=\{a\langle X,Y\rangle+b\eta(X)\eta(Y)\}\nu_{S},
\end{align}
with $a=\cot r$ and $b=-\tan r$. We refer to \cite{BV} for details of the above facts. 

By  the Gauss equation and \eqref{eq:umb},  it turns out that the Ricci tensor of $M=S^3_{r}$ is given by (see also \cite[p.359]{BV})
\begin{align}\label{eq:ricS}
{\rm Ric}^{M}(X,Y)=(2a^{2}+4)\langle X,Y\rangle-4 \eta(X)\eta(Y)
\end{align}
for any $X,Y\in T_{p}S_{r}$. If $X$ is a unit tangent vector, then we see
\begin{align*}
{\rm Ric}^{M}(X,X)=(2a^{2}+4)-4(1-|X_{h}|^{2})=2a^{2}+4|X_{h}|^{2}>0,
\end{align*}
where we set $X=X_{h}-\eta(X)J\nu_{S}$. Therefore, any geodesic hypersphere $S_{r}$ has positive Ricci curvature. Moreover, if $r\neq r'$, then $S_{r}$ and $S_{r'}$ are neither isometric nor homothetic. 

Now, let us consider a closed minimal surface $\Sigma$ in $S_{r}^{3}$. We remark that Torralbo \cite{Tora} proved that there exists an embedded closed minimal surface of arbitrary genus in the Berger sphere. Moreover,  any embedded surface  in $S_r^3$ is two-sided and any closed minimal surface embedded in $S^3_{r}$ is unstable since $S_r^3$ is simply-connected and has positive Ricci curvature. 
Applying Proposition \ref{prop:tr1}, we obtain the following result:

\begin{theorem}\label{thm:a1}
Let $S_{r}^{3}$ be the  geodesic hypersphere of radius $r\in (0,\pi/2)$ in $\C P^{2}(4)$. If the geodesic radius $r$ satisfies
\[
0<\tan r<\sqrt{1+\sqrt{3}}=1.652\cdots,
\]
then, it holds that
\[
{\rm index}(\Sigma)\geq \frac{1}{8}b_{1}(\Sigma)
\]
for any closed minimal surface $\Sigma$ embedded in $S_{r}^{3}$.
\end{theorem}

\begin{proof}
We take $M=S_{r}^{3}$ and compute the trace formula \eqref{eq:trr3} given in Proposition \ref{prop:tr1}. 

For a unit tangent vector $T\in T_{p}\Sigma$, we consider the following quantity:
\begin{align*}
\beta(T):=-{\rm Ric}^{M}(\nu,\nu)+\sum_{\mu=1}^{2}\Big(2|B(\p_{\mu},T)|^{2}-\langle  B(\p_{\mu},\p_{\mu}), B(T,T)\rangle\Big)
\end{align*}
By \eqref{eq:trr3}, if $\beta(T)<0$ for any unit vector $T$ and $p\in \Sigma$, then we have ${\rm Tr}_{\fg}q_{\omega^{\sharp}}<0$. Thus, the result follows from Proposition \ref{prop:meth1} since ${\rm dim}G={\rm dim}SU(3)=8$.

By using \eqref{eq:umb}, we have
\begin{align*}
\sum_{\mu=1}^{2}\Big{(}2|B(\p_{\mu},T)|^{2}-\langle B(\p_{\mu},\p_{\mu}), B(T,T)\rangle\Big{)}
=1-\eta(\nu)^{2}-b^{2}\eta(T)^{2}\eta(\nu)^{2}+(b^{2}-2)\eta(T)^{2}.
\end{align*}
where we used the relation $\sum_{\mu=1}^{2}\eta(\p_{\mu})^{2}=1-\eta(\nu)^{2}$, $ab=-1$ and $|T|^{2}=1$. Moreover, by  \eqref{eq:ricS}, we have
\[
{\rm Ric}^{M}(\nu,\nu)=(2a^{2}+4)-4\eta(\nu)^{2}.
\]
Therefore, we see
\[
\beta(T)=-2a^{2}-3(1-\eta(\nu)^{2})-b^{2}\eta(T)^{2}\eta(\nu)^{2}+(b^{2}-2)\eta(T)^{2}.
\]
%
Since $\eta(T)^{2}=\langle J\nu_{S},T\rangle^{2}\leq |J\nu_{S}|^{2}|T|^{2}\leq 1$ for any unit vector $T$,  we see
\[
\beta(T)\leq -2a^{2}+b^{2}-2=-2\cot^{2}r+\tan^{2}r-2.
\]
Moreover, the inequality $-2\cot^{2}r+\tan^{2}r-2<0$ satisfies if $0<\tan^{2}r<1+\sqrt{3}$. Thus, if this is the case, we have $\beta(T)<0$ for any $T\in T_{p}\Sigma$ and $p\in \Sigma$, and the desired conclusion holds.
\end{proof}

\begin{remark}
{\rm 
By the result of \cite{Savo}, we have ${\rm index}(\Sigma)\geq Cb_{1}(\Sigma)$ for any closed minimal hypersurface $\Sigma$ in  the standard sphere $(S^{3},g_{st})$. Moreover, it follows from \cite[Theorem A]{GMR} that there exists an $\epsilon>0$ such that for any smooth metric $g'$ on $S^{3}$ satisfying that $\|g_{st}-g'\|_{C^{2,\lambda}}<\epsilon$, $\lambda\in (0,1)$,  a similar index estimate holds even in $(S^{3},g')$. Since the Berger sphere is actually a deformation of the standard sphere, one may obtain the index estimate in the Berger sphere by using \cite[Theorem A]{GMR} if the Berger metric is sufficiently close to the standard metric.  
}
\end{remark}


\section{Trace formula II}\label{sec:tr2}
Continued from the previous section, let $N^{n+k}=G/K$ be a compact semi-simple RSS and suppose that a Riemannian manifold $M^{n}$ is isometrically immersed into $N^{n+k}$ via the map $F: M\to N$.  Note that we can assume that $M=N$ and $F$ is the identity map. Let $f:\Sigma^{n-1}\to M^{n}$ be a minimal immersion, and throughout  this section, we suppose that $f$ has a trivial normal bundle. In this section, we extend the trace formula \cite[Proposition 2]{ACS} to  compact semi-simple symmetric spaces.

\subsection{Trace formula}
 We denote the $\R$-linear space of exterior algebra  of second degree on the Lie algebra $\fg$ by $\wedge^{2}\fg$, namely, $\wedge^{2}\fg$ is a vector space spanned by elements of the form $X\wedge Y$ for $X,Y\in \fg$. For an element  $v\in \wedge^{2}\fg$, we define a smooth function on $\Sigma$ as follows: By using the inner product $\langle,\rangle_{\fg}$ on $\fg$,  we define a positive-definite inner product on $\wedge^{2}\fg$  by
\[
\langle X\wedge Y, Z\wedge W\rangle_{\wedge^{2}\fg}:=\langle X,Z\rangle_{\fg}\langle  Y,W\rangle_{\fg}-\langle X,W\rangle_{\fg}\langle  Y,Z\rangle_{\fg}
\]
for elements $X\wedge Y, Z\wedge W\in \wedge^{2}\fg$ and we linearly extend it to whole $\wedge^{2}\fg$.

Let $\{\theta_{i}\}_{i=1}^{d}$ be a basis of $\fg$. Then the associated basis of $\wedge^{2}\fg$ is given by $\{\theta_{i}\wedge \theta_{j}\}_{1\leq i<j\leq d}$. For each $p\in \Sigma$, we define a linear map $\Pi: \wedge^{2}\fg\to\Gamma(\wedge^{2}TN)$ by
\[\Pi(\theta_{i}\wedge \theta_{j})(p)=\theta_{i}^{\dagger}(p)\wedge \theta_{j}^{\dagger}(p)\]
 and we linearly extend the map to whole $\wedge^{2}\fg$. Note that the definition of $\Pi$ does not depend on the choice of  basis of $\fg$.

 Now, we consider a minimal immersion $f:\Sigma^{n-1}\to M^{n}$ with trivial normal bundle.  We fix a unit normal vector field $\nu$ along $f$. Then, we consider a map
 \[
\Phi: \Gamma(T\Sigma)\times \wedge^2\fg\to \Gamma(\nu\Sigma),\quad \Phi(T,v):=\varphi_{T,v}\cdot \nu,
\]
where $\varphi_{T,v}\in C^{\infty}(\Sigma)$ is defined by 
\begin{align}\label{def:phi}
\varphi_{T,v}=\langle \nu\wedge T, \Pi(v)\rangle.
\end{align}
Also, we consider the quadratic form on $\wedge^{2}\fg$ by
\[
q_{T}(v,w):=Q(\varphi_{T,v},\varphi_{T,w})
\]
for $T\in \Gamma(T\Sigma)$.
The aim of this section is to compute the trace of $q_{T}$ over $\wedge^{2}\fg$.

\begin{remark}
{\rm If $M=N$ (i.e. $F=id_{M}$), it turns out that the test function \eqref{def:phi} coincides with the one considered in \cite{MR}. See Appendix \ref{A:MR} for the detail.}
\end{remark}

For each $p\in \Sigma$, we define a quadratic form $q_{T,p}$ on $\wedge^{2}\fg$ by
\[
q_{T,p}(v,w):=\langle \mathcal{J}(\varphi_{T,v}), \varphi_{T,w}\rangle(p)
\]
so that 
\begin{align}\label{eq:trtr}
{\rm Tr}_{\wedge^{2}\fg} q_{T}=-\int_{\Sigma}{\rm Tr}_{\wedge^{2}\fg}q_{T,p}\,d\mu_{\Sigma}.
\end{align}
In the following, we abbreviate the fixed index $T$, e.g. $q_{\theta_{i}}=q_{T,\theta_{i}}$ etc.

\begin{proposition}\label{prop:tr}
 For any smooth tangent vector filed $T$ of $\Sigma$, we have
\begin{align}\label{eq:trq0}
{\rm Tr}_{\wedge^{2}\fg} q_{T,p}&=-(\Delta_{1}T^{\flat})(T)\\
&\quad -\sum_{\mu=1}^{n-1}\Big(|B(\p_{\mu},T)|^{2}+|B(\p_{\mu},\nu)|^{2}|T|^{2}\Big)+\sum_{\mu=1}^{n-1}\langle {R}^{M}(T,\p_{\mu})\p_{\mu}, T\rangle+{\rm Ric}^{M}(\nu,\nu)|T|^{2}\nonumber\\
&\quad -\sum_{\mu=1}^{n-1}\Big(\langle \overline{R}(T,\p_{\mu})\p_{\mu}, T\rangle+\langle \overline{R}(\nu, \p_{\mu})\p_{\mu},\nu\rangle|T|^{2}\Big),\nonumber
\end{align}
where $T^{\flat}:=\langle T, \cdot \rangle|_{\Sigma}$, $\Delta_{1}$ is the Hodge Laplacian acting on $\Omega^{1}(\Sigma)$.
\end{proposition}

\begin{proof}
 By Lemma \ref{lem:key2}, the decomposition $\fg=\fk_{p}\oplus \fn_{p}$ is orthogonal w.r.t. $\langle,\rangle_{\fg}$,  and hence, we can take an orthonormal basis $\{X_{1},\ldots X_{d-(n+k)}, Y_{1},\ldots, Y_{n+k}\}$ of $\fg$ so that $\{X_{1},\ldots, X_{d-(n+k)}\}$ is a basis of $\fk_{p}$ and $\{Y_{1},\ldots, Y_{n+k}\}$ is a basis of $\fn_{p}$. Then, the orthonormal basis of $\wedge^{2}\fg$ is given by $X_{l}\wedge X_{m}, X_{l}\wedge Y_{i}, Y_{i}\wedge Y_{j}$ for $1\leq l<m\leq d-(n+k)$ and $1\leq i<j\leq n+k$.
Since $X_{l}^{\dagger}(p)=0$ for any $l=1,\ldots, d-(n+k)$ by Lemma \ref{lem:key1}, we see
\begin{align*}
\varphi_{X_{l}\wedge X_{m}}(p)=\varphi_{X_{l}\wedge Y_{i}}(p)=0 
\end{align*}
for any $l,m=1,\ldots, d-(n+k)$ and $i=1,\ldots, n+k$. In particular, we have
\begin{align}\label{eq:tr0}
{\rm Tr}_{\wedge^{2}\fg}q_{p}&=\sum_{1\leq i<j\leq n+k}  q_{p}(Y_{i}\wedge Y_{j}, Y_{i}\wedge Y_{j})
\end{align}
Recall that, under the identification $\fn_{p}\simeq T_{p}N$, $\{Y_{1}^{\dagger}(p),\ldots, Y_{n+k}^{\dagger}(p)\}$ becomes an orthonormal basis of $T_{p}N$. We also take a geodesic normal coordinate of the hypersurface $\Sigma$ around $p$, and denote the coordinate basis by $\{\p_{\mu}\}_{\mu=1}^{n-1}$. 

By \eqref{eq:tr0}, we have
\begin{align}\label{eq:j1}
{\rm Tr}_{\wedge^{2}\fg}q_{p}
&=\sum_{1\leq i<j\leq n+k} \Delta\varphi_{Y_{i}\wedge Y_{j}}\cdot \varphi_{Y_{i}\wedge Y_{j}}(p)+({\rm Ric}^{M}(\nu,\nu)+|A|^{2})\varphi_{Y_{i}\wedge Y_{j}}^{2}(p).
\end{align}

Note that we  have
\begin{align}\label{eq:lem}
&\sum_{1\leq i<j\leq n+k}\langle Z\wedge W, Y_{i}^{\dagger}\wedge Y_{j}^{\dagger}\rangle\cdot \varphi_{Y_{i}\wedge Y_{j}} (p)=\langle Z,\nu\rangle\langle W,T\rangle-\langle Z,T\rangle\langle W,\nu\rangle
\end{align}
 for any $Z,W\in T_{p}N$.
Indeed, we see
\begin{align}\label{eq:lem1}
&\langle Z\wedge W, Y_{i}^{\dagger}\wedge Y_{j}^{\dagger}\rangle\cdot \varphi_{Y_{i}\wedge Y_{j}}\\
&=\Big(\langle Z, Y_{i}^{\dagger}\rangle\langle W,Y_{j}^{\dagger}\rangle-\langle Z, Y_{j}^{\dagger}\rangle\langle W,Y_{i}^{\dagger}\rangle\Big)
\Big(\langle \nu, Y_{i}^{\dagger}\rangle\langle T,Y_{j}^{\dagger}\rangle-\langle \nu, Y_{j}^{\dagger}\rangle\langle T,Y_{i}^{\dagger}\rangle\Big)\nonumber\\
&=\langle Z, Y_{i}^{\dagger}\rangle\langle W,Y_{j}^{\dagger}\rangle\langle \nu, Y_{i}^{\dagger}\rangle\langle T,Y_{j}^{\dagger}\rangle
-\langle Z, Y_{i}^{\dagger}\rangle\langle W,Y_{j}^{\dagger}\rangle \langle \nu, Y_{j}^{\dagger}\rangle\langle T,Y_{i}^{\dagger}\rangle\nonumber\\
&\quad-\langle Z, Y_{j}^{\dagger}\rangle\langle W,Y_{i}^{\dagger}\rangle\langle \nu, Y_{i}^{\dagger}\rangle\langle T,Y_{j}^{\dagger}\rangle
+\langle Z, Y_{j}^{\dagger}\rangle\langle W,Y_{i}^{\dagger}\rangle\langle \nu, Y_{j}^{\dagger}\rangle\langle T,Y_{i}^{\dagger}\rangle. \nonumber
\end{align}
On the other hand, we have 
\begin{align}\label{eq:lem2}
&\sum_{1\leq i<j\leq n+k}\langle Z\wedge W, Y_{i}^{\dagger}\wedge Y_{j}^{\dagger}\rangle\cdot \varphi_{Y_{i}\wedge Y_{j}} (p)
=\frac{1}{2}\sum_{i=1}^{n+k}\sum_{j=1}^{n+k}\langle Z\wedge W, Y_{i}^{\dagger}\wedge Y_{j}^{\dagger}\rangle\cdot \varphi_{Y_{i}\wedge Y_{j}} (p)
\end{align}
since $\langle Z\wedge W, Y_{i}^{\dagger}\wedge Y_{j}^{\dagger}\rangle\cdot \varphi_{Y_{i}\wedge Y_{j}}=\langle Z\wedge W, Y_{j}^{\dagger}\wedge Y_{i}^{\dagger}\rangle\cdot \varphi_{Y_{j}\wedge Y_{i}}$ and $Y_{i}^{\dagger}\wedge Y_{i}^{\dagger}=0$.
Substituting \eqref{eq:lem1} to \eqref{eq:lem2} and by using the fact that $\{Y_{i}^{\dagger}(p)\}_{i}$ is an orthonormal basis of $T_{p}N$, we obtain \eqref{eq:lem}.

In particular, in the last term of \eqref{eq:j1}, we have
\begin{align}\label{eq:j12}
\sum_{1\leq i<j\leq n+k} \varphi_{Y_{i}\wedge Y_{j}}^{2}=|\nu|^{2}|T|^{2}-\langle \nu,T\rangle^{2}=|T|^{2}
\end{align}
since $|\nu|^{2}=1$ and $\langle \nu,T\rangle=0$.

We shall consider the first term of the RHS of \eqref{eq:j1}. 
Since $\onab Y_{i}^{\dagger}(p)=0$ for any $i$ by Lemma \ref{lem:key1}, we have $\onab(Y_{i}^{\dagger}\wedge Y_{j}^{\dagger})(p)=\onab Y_{i}^{\dagger}\wedge Y_{j}^{\dagger}(p)+Y_{i}^{\dagger}\wedge \onab Y_{j}^{\dagger}(p)=0$. Therefore, we see
\begin{align}
&\Delta\varphi_{Y_{i}\wedge Y_{j}}(p)\cdot \varphi_{Y_{i}\wedge Y_{j}}(p)\nonumber\\
&=\sum_{\mu=1}^{n-1}\p_{\mu}\p_{\mu}\langle \nu\wedge T, Y_{i}^{\dagger}\wedge Y_{j}^{\dagger}\rangle(p)\cdot \varphi_{Y_{i}\wedge Y_{j}}(p)\nonumber\\
\label{eq:lap}
&=\sum_{\mu=1}^{n-1}\langle \onab_{\p_{\mu}}\onab_{\p_{\mu}}(\nu\wedge T), Y_{i}^{\dagger}\wedge Y_{j}^{\dagger}\rangle(p)\cdot \varphi_{Y_{i}\wedge Y_{j}}(p)
+\sum_{\mu=1}^{n-1}\langle \nu\wedge T, \onab_{\p_{\mu}}\onab_{\p_{\mu}}(Y_{i}^{\dagger}\wedge Y_{j}^{\dagger})\rangle(p)\cdot \varphi_{Y_{i}\wedge Y_{j}}(p),
\end{align}

We consider the first term  of \eqref{eq:lap}. We have
\begin{align}\label{eq:lap2}
\onab_{\p_{\mu}}\onab_{\p_{\mu}}(\nu\wedge T)
=(\onab_{\p_{\mu}}\onab_{\p_{\mu}}\nu)\wedge T+2\onab_{\p_{\mu}}\nu\wedge \onab_{\p_{\mu}}T+\nu\wedge (\onab_{\p_{\mu}}\onab_{\p_{\mu}}T),
\end{align}
and hence, by \eqref{eq:lem}, we obtain
\begin{align*}
&\sum_{1\leq i<j\leq n+k} \langle \onab_{\p_{\mu}}\onab_{\p_{\mu}}(\nu\wedge T), Y_{i}^{\dagger}\wedge Y_{j}^{\dagger}\rangle\cdot \varphi_{Y_{i}\wedge Y_{j}} (p)\\
&=\langle \onab_{\p_{\mu}}\onab_{\p_{\mu}}\nu,\nu\rangle |T|^{2}+2\langle \onab_{\p_{\mu}}\nu,T\rangle^{2}
+\langle \onab_{\p_{\mu}}\onab_{\p_{\mu}}T,T\rangle.
\end{align*}
where we used the facts $\langle \nu, T\rangle=0$ and  $|\nu|^{2}=1$ along $\Sigma$ which also imply $\langle \onab_{\p_{\mu}}\nu,T\rangle=-\langle \nu, \onab_{\p_{\mu}}T\rangle$ and $\langle \onab_{\p_{\mu}}\nu,\nu\rangle=0$. 
Moreover, we have
\begin{align*}
\sum_{\mu=1}^{n-1}\langle \onab_{\p_{\mu}}\onab_{\p_{\mu}}\nu,\nu\rangle |T|^{2}
&=\sum_{\mu=1}^{n-1}-|\onab_{\p_{\mu}}\nu|^{2} |T|^{2}\quad (\textup{since $\langle \onab_{\p_{\mu}}\nu,\nu\rangle=0$ along $\Sigma$})\\
&=\sum_{\mu=1}^{n-1}(-|\nabla^{M}_{\p_{\mu}}\nu|^{2}-|B(\p_{\mu}, \nu)|^{2})|T|^{2}
=-|A|^{2}|T|^{2}-\sum_{\mu=1}^{n-1}|B(\p_{\mu}, \nu)|^{2}|T|^{2}.\\
2\sum_{\mu=1}^{n-1}\langle \onab_{\p_{\mu}}\nu,T\rangle^{2}
&=2\sum_{\mu=1}^{n-1}\langle \nabla^{M}_{\p_{\mu}}\nu,T\rangle^{2}
=2\sum_{\mu=1}^{n-1}|A(\p_{\mu},T)|^{2}.\\
 \sum_{\mu=1}^{n-1}\langle \onab_{\p_{\mu}}\onab_{\p_{\mu}}T,T\rangle
& =\frac{1}{2}\Delta |T|^{2}-\sum_{\mu=1}^{n-1}|\onab_{\p_{\mu}}T|^{2}\\
&=\frac{1}{2}\Delta |T|^{2}-|\nabla T|^{2}-\sum_{\mu=1}^{n-1}|A(\p_{\mu},T)|^{2}-\sum_{\mu=1}^{n-1}|B(\p_{\mu},T)|^{2}.
\end{align*}
Summing up these, we obtain
\begin{align}\label{eq:lap1}
&\sum_{1\leq i<j\leq m}\sum_{\mu=1}^{n-1} \langle \onab_{\p_{\mu}}\onab_{\p_{\mu}}(\nu\wedge T), Y_{i}^{\dagger}\wedge Y_{j}^{\dagger}\rangle\cdot \varphi_{Y_{i}\wedge Y_{j}} (p)\\
&=\frac{1}{2}\Delta |T|^{2}-|\nabla T|^{2}+\sum_{\mu=1}^{n-1}|A(\p_{\mu},T)|^{2}-|A|^{2}|T|^{2}
-\sum_{\mu=1}^{n-1}\Big(|B(\p_{\mu},T)|^{2}+|B(\p_{\mu},\nu)|^{2}|T|^{2}\Big).\nonumber
\end{align}

Next, we consider the second term of \eqref{eq:lap}.
Since $\onab Y_{i}^{\dagger}(p)=0$, we have
\begin{align*}
\onab_{\partial_{\mu}}\onab_{\partial_{\mu}}(Y_{i}^{\dagger}\wedge Y_{j}^{\dagger})(p)
=\{(\onab_{\partial_{\mu}}\onab_{\partial_{\mu}}Y_{i}^{\dagger})\wedge Y_{j}^{\dagger}+Y_{i}^{\dagger}\wedge (\onab_{\partial_{\mu}}\onab_{\partial_{\mu}}Y_{j}^{\dagger})\}(p).
\end{align*}
Moreover, by \eqref{eq:kill},
we obtain 
\begin{align}\label{eq:nn}
\onab_{\partial_{\mu}}\onab_{\partial_{\mu}}(Y_{i}^{\dagger}\wedge Y_{j}^{\dagger})(p)
=-(\overline{R}(Y_{i}^{\dagger},\p_{\mu})\p_{\mu})\wedge Y_{j}^{\dagger}(p)+(\overline{R}(Y_{j}^{\dagger},\p_{\mu})\p_{\mu})\wedge Y_{i}^{\dagger}(p).
\end{align}
Since $\varphi_{Y_{i}\wedge Y_{j}}=-\varphi_{Y_{j}\wedge Y_{i}}$, the second term of \eqref{eq:lap} becomes
\begin{align}\label{eq:lap20}
&\sum_{1\leq i<j\leq n+k} \sum_{\mu=1}^{n-1}\langle \nu\wedge T, \onab_{\p_{\mu}}\onab_{\p_{\mu}}(Y_{i}^{\dagger}\wedge Y_{j}^{\dagger})\rangle\cdot \varphi_{Y_{i}\wedge Y_{j}} (p)\\
&=\sum_{\mu=1}^{n-1}\sum_{i=1}^{n+k}\sum_{j=1}^{n+k}-\langle \nu\wedge T, \overline{R}(Y_{i}^{\dagger},\p_{\mu})\p_{\mu}\wedge Y_{j}^{\dagger}\rangle \cdot \varphi_{Y_{i}\wedge Y_{j}} (p). \nonumber
\end{align}
Here, we have
\begin{align*}
\langle \nu\wedge T, \overline{R}(Y_{i}^{\dagger},\p_{\mu})\p_{\mu}\wedge Y_{j}^{\dagger}\rangle
&=
\langle \nu, \overline{R}(Y_{i}^{\dagger},\p_{\mu})\p_{\mu}\rangle\langle T, Y_{j}^{\dagger}\rangle
-\langle T, \overline{R}(Y_{i}^{\dagger},\p_{\mu})\p_{\mu}\rangle\langle \nu, Y_{j}^{\dagger}\rangle\\
\varphi_{Y_{i}\wedge Y_{j}}
&=\langle \nu, Y_{i}^{\dagger}\rangle\langle T, Y_{j}^{\dagger}\rangle-\langle T, Y_{i}^{\dagger}\rangle\langle \nu, Y_{j}^{\dagger}\rangle
\end{align*}
and hence, 
\begin{align*}
&-\langle \nu\wedge T, \overline{R}(Y_{i}^{\dagger},\p_{\mu})\p_{\mu}\wedge Y_{j}^{\dagger}\rangle \cdot \varphi_{Y_{i}\wedge Y_{j}}\\
&=-\langle \nu, \overline{R}(Y_{i}^{\dagger},\p_{\mu})\p_{\mu}\rangle\langle \nu, Y_{i}^{\dagger}\rangle\langle T, Y_{j}^{\dagger}\rangle^{2}
+\langle \nu, \overline{R}(Y_{i}^{\dagger},\p_{\mu})\p_{\mu}\rangle\langle \nu, Y_{j}^{\dagger}\rangle\langle T, Y_{i}^{\dagger}\rangle\langle T, Y_{j}^{\dagger}\rangle\\
&\quad +\langle T, \overline{R}(Y_{i}^{\dagger},\p_{\mu})\p_{\mu}\rangle\langle \nu, Y_{i}^{\dagger}\rangle\langle \nu, Y_{j}^{\dagger}\rangle\langle T, Y_{j}^{\dagger}\rangle
-\langle T, \overline{R}(Y_{i}^{\dagger},\p_{\mu})\p_{\mu}\rangle\langle \nu, Y_{j}^{\dagger}\rangle^{2}\langle T, Y_{i}^{\dagger}\rangle.
\end{align*}
Summing up for $i,j=1,\ldots, n+k$, \eqref{eq:lap20} shows that
\begin{align}\label{eq:lap2}
&\sum_{1\leq i<j\leq n+k} \sum_{\mu=1}^{n-1}\langle \nu\wedge T, \onab_{\p_{\mu}}\onab_{\p_{\mu}}(Y_{i}^{\dagger}\wedge Y_{j}^{\dagger})\rangle\cdot \varphi_{Y_{i}\wedge Y_{j}} (p)\\
&=\sum_{\mu=1}^{n-1}\Big(-\langle \overline{R}(\nu, \p_{\mu})\p_{\mu},\nu\rangle|T|^{2}-\langle \overline{R}(T,\p_{\mu})\p_{\mu}, T\rangle\Big).\nonumber
\end{align}
since $\langle T,\nu\rangle=0$ and $|\nu|^{2}=1$.

By \eqref{eq:j1}, \eqref{eq:j12}, \eqref{eq:lap}, \eqref{eq:lap1} and \eqref{eq:lap2}, we obtain 
\begin{align}\label{eq:trq1}
{\rm Tr}_{\wedge^{2}\fg} q_{p}&=\frac{1}{2}\Delta |T|^{2}-|\nabla T|^{2}+\sum_{\mu=1}^{n-1}|A(\p_{\mu},T)|^{2}
-\sum_{\mu=1}^{n-1}\Big(|B(\p_{\mu},T)|^{2}+|B(\p_{\mu},\nu)|^{2}|T|^{2}\Big)\\
&\quad -\sum_{\mu=1}^{n-1}\Big(\langle \overline{R}(T,\p_{\mu})\p_{\mu}, T\rangle+\langle \overline{R}(\nu, \p_{\mu})\p_{\mu},\nu\rangle|T|^{2}\Big)+{\rm Ric}^{M}(\nu,\nu)|T|^{2}.\nonumber
\end{align}

Finally, substituting the identity \eqref{eq:BW} to \eqref{eq:trq1}, the term $\sum_{\mu=1}^{n-1}| A(\p_{\mu},T)|^2$ in \eqref{eq:trq1} is cancelled out, and we obtain the desired formula \eqref{eq:trq0}.
\end{proof}

Now, we consider a special situation. Let $\omega$ be a harmonic $1$-form on $\Sigma$. We take $T=\omega^{\sharp}$ the metric dual of $\omega$. Then, we have $\Delta_{1}T^{\flat}=\Delta_{1}\omega=0$, and hence, by \eqref{eq:trtr}--\eqref{eq:trq0}, we obtain the following trace formula.

\begin{theorem}\label{thm:tr}
Suppose a Riemannian manifold $M^{n}$ is isometrically immersed into a compact semi-simple Riemannian symmetric space $N^{n+k}$, and  $f: \Sigma^{n-1}\to M^{n}$ is a minimal immersion with trivial normal bundle. Then, we have
\begin{align}
\label{eq:tr1}
{\rm Tr}_{\wedge^{2}\fg} q_{\omega^\sharp}
&=\int_{\Sigma}\sum_{\mu=1}^{n-1}\Big(|B(e_{\mu},\omega^{\sharp})|^{2}+|B(e_{\mu},\nu)|^{2}|\omega|^{2}\Big)
-\Big(\sum_{\mu=1}^{n-1}\langle R^{M}(\omega^{\sharp},e_{\mu})e_{\mu}, \omega^{\sharp}\rangle+{\rm Ric}^{M}(\nu,\nu)|\omega|^{2}\Big)\\
&\qquad +\sum_{\mu=1}^{n-1}\Big(\langle \overline{R}(\omega^{\sharp},e_{\mu})e_{\mu}, \omega^{\sharp}\rangle+\langle \overline{R}(\nu, e_{\mu})e_{\mu},\nu\rangle|\omega|^{2}\Big)\,d\mu_{\Sigma},\nonumber
\end{align}
for any harmonic $1$-form $\omega$ on $\Sigma$, where $\{e_{\mu}\}_{\mu=1}^{n-1}$ is an orthonormal basis of $T_{p}\Sigma$.
\end{theorem}

By Proposition \ref{prop:ind}-(i), we have the following result which generalizes \cite[Theorem A]{ACS}.
\begin{proposition}\label{prop:meth2}
Let $M^{n}$ be a Riemannian manifold and $\Sigma$ a two-sided closed minimal hypersurface $\Sigma$ in $M$. If there is an isometric embedding $F: M^{n}\to N^{n+k}$ into a compact semi-simple Riemannian symmetric space $N=G/K$ so that ${\rm Tr}_{\wedge^{2}\fg}q_{\omega^{\sharp}}<0$ for any non-trivial harmonic $1$-form $\omega$ on $\Sigma$, then we have 
\[
{\rm index}(\Sigma)\geq \frac{2}{d(d-1)}b_{1}(\Sigma),
\]
where $d:={\rm dim}G={\rm dim}{\rm Isom}_{0}(N)$.
\end{proposition}

We shall mention another important fact. By the Gauss equation, we have 
\begin{align*}
\sum_{\mu=1}^{n-1}\Big(-\langle R^{M}(\omega^{\sharp},e_{\mu})e_{\mu}, \omega^{\sharp}\rangle+\langle \overline{R}(\omega^{\sharp},e_{\mu})e_{\mu}, \omega^{\sharp}\rangle\Big)
&=\sum_{\mu=1}^{n-1}\Big(|B(\omega^{\sharp},e_{\mu})|^{2}-\langle B(\omega^{\sharp},\omega^{\sharp}), B(e_{\mu},e_{\mu})\rangle\Big),\\
-{\rm Ric}^{M}(\nu,\nu)+\sum_{\mu=1}^{n-1}\langle\overline{R}(\nu, e_{\mu})e_{\mu},\nu\rangle&=\sum_{\mu=1}^{n-1}\Big(|B(\nu,e_{\mu})|^{2}-\langle B(\nu,\nu), B(e_{\mu},e_{\mu})\rangle\Big).
\end{align*}
Thus, the trace formula \eqref{eq:tr1} is written as
\begin{align}
\label{eq:tr2}
{\rm Tr}_{\wedge^{2}\fg} q_{\omega^\sharp}&=\int_{\Sigma}\sum_{\mu=1}^{n-1}2\Big(|B(\omega^{\sharp},e_{\mu})|^{2}+|B(\nu,e_{\mu})|^{2}|\omega|^{2}\Big)
\\
&\qquad -\sum_{\mu=1}^{n-1}\Big(\langle B(\omega^{\sharp},\omega^{\sharp}), B(e_{\mu},e_{\mu})\rangle+\langle B(\nu,\nu), B(e_{\mu},e_{\mu})\rangle|\omega|^{2}\Big)\,d\mu_{\Sigma} \nonumber
\end{align}
In particular,  we obtain the following  consequence.

\begin{proposition}\label{prop:key2}
Let $M$ be a totally geodesic submanifold (possibly $M=N$) in a compact semi-simple Riemannian symmetric space $N$, and $f: \Sigma^{n-1}\to M^{n}$ be a minimal immersion with trivial normal bundle. Then we have ${\rm Tr}_{\wedge^{2}\fg}q_{\omega^{\sharp}}=0$ for any harmonic $1$-form $\omega$ on $\Sigma$.
\end{proposition}

We remark that any complete totally geodesic submanifold in a symmetric space is also a symmetric space (but not necessarily semi-simple). 

Combining Proposition \ref{prop:key2} with Proposition \ref{prop:ind}-(ii),  we obtain the following theorem which has been proved in \cite[Theorem A]{MR} (However, our approach is different from \cite{MR}).

\begin{theorem}[Mendes-Radeschi \cite{MR}]
Let $\Sigma$ be a closed, two-sided minimal hypersurface in a compact semi-simple symmetric space $N=G/K$ of ${\rm dim}\ G=d$. Then we have
\[
{\rm index}(\Sigma)+{\rm nullity}(\Sigma)\geq \frac{2}{d(d-1)}b_{1}(\Sigma).
\]
\end{theorem}

\begin{remark}[Comparison with the trace formula in $\R^{m}$]
{\rm 
Under the same assumption given in Theorem \ref{thm:tr}, we further assume that the semi-simple compact RSS $N$ is isometrically immersed into the Euclidean space $\R^{m}$ via the map $G: N\to \R^{m}$.  Namely, we suppose that we have the following isometric immersions:
\[
\Sigma^{n-1}\xrightarrow{f} M^{n}\xrightarrow{F} N^{n+k}\xrightarrow{G} \R^{m}.
\]
Note that, for a symmetric space $N$, there are some well-studied isometric immersions from $N$ into $\R^{m}$ (e.g. see \cite{ACS, GMR}). 
We denote the second fundamental form of three isometric immersions $F: M\to N$,  $G\circ F: M\to \R^{m}$ and $G: N\to \R^{m}$ by $B$, $B'$ and $B_{N}$, respectively.

Since $M^{n}$ is isometrically immersed into $\R^{m}$ by the map $G\circ F$, we can apply the tracing method  used in \cite{ACS} to the map $G\circ F$, and then, we have the following trace formula (see \eqref{eq:acs} in Appendix):
\begin{align}\label{eq:tr22}
{\rm Tr}_{\wedge^{2}\R^{m}} \widehat{q}_{\omega^{\sharp}}&=\int_{\Sigma}\sum_{\mu=1}^{n-1}2\Big(|B'(\omega^{\sharp},e_{\mu})|^{2}+|B'(\nu,e_{\mu})|^{2}|\omega|^{2}\Big)
\\
&\qquad -\sum_{\mu=1}^{n-1}\Big(\langle B'(\omega^{\sharp},\omega^{\sharp}), B'(e_{\mu},e_{\mu})\rangle+\langle B'(\nu,\nu), B'(e_{\mu},e_{\mu})\rangle|\omega|^{2}\Big)\,d\mu_{\Sigma},\nonumber
\end{align}
where $\widehat{q}_{\omega^{\sharp}}$ is a quadratic form on $\wedge^{2}\R^{m}$ considered in \cite{ACS} (see Appendix for details).
Notice that the right hand side  is nothing but the replacing $B$ with $B'$ in  \eqref{eq:tr2}.
Since $B'(X,Y)=B(X,Y)+B_{N}(X,Y)$ for any $X,Y\in T_{p}M$, and $B(X,Y)$ is orthogonal to $B_{N}(X,Y)$, the formulas \eqref{eq:tr2} and  \eqref{eq:tr22} show the following relation:
\begin{align}\label{eq:com}
{\rm Tr}_{\wedge^{2}\R^{m}} \widehat{q}_{\omega^{\sharp}}&={\rm Tr}_{\wedge^{2}\fg}q_{\omega^{\sharp}}+\int_{\Sigma} \beta_{N}\Big(\frac{\omega^{\sharp}}{|\omega|},\nu\Big)\cdot |\omega|^{2}\,d\mu_{\Sigma},
\end{align}
where we put
\begin{align*}
\beta_{N}(X,\nu)&:=\sum_{\mu=1}^{n-1}2\Big(|B_{N}(X,e_{\mu})|^{2}+|B_{N}(\nu,e_{\mu})|^{2}\Big)
\\
&\quad -\sum_{\mu=1}^{n-1}\Big(\langle B_{N}(X,X), B_{N}(e_{\mu},e_{\mu})\rangle+\langle B_{N}(\nu,\nu), B_{N}(e_{\mu},e_{\mu})\rangle \Big)\nonumber
\end{align*}
for a unit tangent vector $X\in T_{p}\Sigma$, and $\{e_{\mu}\}_{\mu=1}^{n-1}$ is an orthonormal basis of $T_{p}\Sigma$.

We remark that, by the relation \eqref{eq:com}, if $\beta_{N}(X,\nu)<0$ for any unit vector $X\in T_{p}\Sigma$ and $p\in \Sigma$, then it holds that ${\rm Tr}_{\wedge^{2}\R^{m}} \widehat{q}_{\omega^{\sharp}}<{\rm Tr}_{\wedge^{2}\fg}q_{\omega^{\sharp}}$ for any non-trivial harmonic $1$-form $\omega$. Therefore, in this case, it is more reasonable to use ${\rm Tr}_{\wedge^{2}\R^{m}} \widehat{q}_{\omega^{\sharp}}$ rather than ${\rm Tr}_{\wedge^{2}\fg}q_{\omega^{\sharp}}$ to show that the trace is negative (or non-positive) in order to apply Proposition \ref{prop:ind}. For example, if 
$G: S^{n+k}(r)\to \R^{n+k+1}$ is the canonical isometric embedding of the sphere of radius $r$,  then it is easy to see that $\beta_{N}(X,\nu)=-\frac{2}{r^{2}}(n-2)$ and  $\beta_{N}(X,\nu)<0$ if $n\geq 3$.  This means that if $M^{n}$ is isometrically immersed into $N=S^{n+k}$, it seems better to consider $G\circ F: M\to \R^{n+k+1}$ and apply the method used in \cite{ACS} instead of using our trace formula \eqref{eq:tr1} for $F: M^{n}\to S^{n+k}$.

However, we may not expect such a situation for a general symmetric space $N$. Namely, there is a possibility that the opposite inequality holds; ${\rm Tr}_{\wedge^{2}\R^{m}} \widehat{q}_{\omega^{\sharp}}>{\rm Tr}_{\wedge^{2}\fg}q_{\omega^{\sharp}}$. For example, by Proposition \ref{prop:key2}, when $M=N$ and $F$ is the identity map, we always have ${\rm Tr}_{\wedge^{2}\fg}q_{\omega^{\sharp}}=0$ even if $ {\rm Tr}_{\wedge^{2}\R^{m}} \widehat{q}_{\omega^{\sharp}}>0$.}
\end{remark}

\subsection{Applications}

As a generalization of Theorem \ref{thm:a1} concerning on the Berger sphere in $\C P^{2}$,  by using our trace formula given in Theorem \ref{thm:tr},
we shall prove an index estimate for minimal hypersurface in some geodesic hypersphere of the $\mathbb{K}$-projective space, where $\K=\C, \mathbb{H}$ or $\mathbb{O}$. More precisely, we show the following result.
\begin{theorem}\label{thm:a2}
Let $\mathbb{K}P^{m}$ be the $\mathbb{K}$-projective space, where $\mathbb{K}=\C, \mathbb{H}$ (for $m\geq 2$) or $\mathbb{O}$ (with $m=2$), and $S_{r}^{n}$ be the geodesic hypersphere of radius $r\in (0,\pi/2)$ in $\mathbb{K}P^{m}$. Then there exists some positive constant $c_{n}$ depending only on $n$ such that if 
$
0<\tan r<c_{n}
$
then we have 
\[
{\rm index}(\Sigma)\geq \frac{2}{d(d-1)}b_{1}(\Sigma)
\]
for any closed minimal hypersurface $\Sigma$ embedded in $S_{r}^{n}$, where $d={\rm dim}{\rm Isom}_{0}\mathbb{K}P^{m}$.
\end{theorem}

An explicit value of $c_{n}$ is given in the following proof.

\begin{proof}
 It is known that in any case, the geodesic hypersphere $S_{r}^{n}$ has exactly two constant principal curvatures, and we may assume that they are given by $\lambda_{1}=\cot r$ and $\lambda_{2}=2\cot(2r)$ (see \cite{Berndt, BV, Mur}). Moreover the multiplicity $m(\lambda_{i})$ of each principal curvature is given in the table \ref{table}.

\begin{table}[h]
\centering
\begin{tabular}{|c|c|c|c|} \hline
$\K$ & $n$ & $m(\lambda_{1})$ & $m(\lambda_{2})$ \\ \hline
 $\C$ & $2m-1$ &  $2m-2$ & 1 \\ 
 $\mathbb{H}$ & $4m-1$ & $4m-4$ & 3 \\ 
 $\mathbb{O}$ & 15 &  $8$ & $7$ \\\hline
 \end{tabular}
  \caption{Multiplicities of principal curvatures of $S_{r}^{n}$}
  \label{table}
 \end{table}

We take an orthonormal basis $\{U_{1},\ldots, U_{\delta}\}$ of the $\lambda_{2}$-eigenspace, where we put $\delta:=m(\lambda_{2})$. Then it is easy to see that the shape operator $A_{S_{r}}$ of $S_{r}$ in $\mathbb{K}P^{m}$ is given by
\[
A_{S_{r}}(X)=\lambda_{1}X+(\lambda_{2}-\lambda_{1})\sum_{i=1}^{\delta}\eta_{i}(X) U_{i},
\]
where $X$ is a tangent vector of $S_{r}$ and we set $\eta_{i}(X):=\langle U_{i},X\rangle$.
We thus put $a:=\lambda_{1}=\cot r$ and $b:=\lambda_{2}-\lambda_{1}=-\tan r$. Then the second fundamental form of $S_{r}$ is written by
\begin{align}\label{eq:ber}
B(X,Y)=\Big\{a \langle X, Y\rangle+b\sum_{i=1}^{\delta}\eta_{i}(X)\eta_{i}(Y)\Big\}\cdot \nu_{S},
\end{align}
where $\nu_{S}$ is the unit normal vector field of $S_{r}$ in $\mathbb{K}P^{m}$.

Let $\Sigma^{n-1}$ be a closed minimal hypersurface embedded in $S_{r}^{n}$.  Note that $S_{r}^{n}$ is homeomorphic to the $n$-dimensional sphere, and thus, any embedded hypersurface $\Sigma$ in $S_{r}^{n}$ is two-sided. Thus, by taking \eqref{eq:tr2} into account, we consider the following quantity:
\begin{align*}
\gamma(T)&:=\sum_{\mu=1}^{n-1}2\Big(|B(T,e_{\mu})|^{2}+|B(\nu,e_{\mu})|^{2}\Big)\\
&\quad -\sum_{\mu=1}^{n-1} \Big(\langle B(T,T), B(e_{\mu},e_{\mu})\rangle+\langle B(\nu,\nu), B(e_{\mu},e_{\mu})\rangle\Big),
\end{align*}
where $T$ is a unit tangent vector in $T_{p}\Sigma$, $\nu$ is the unit normal vector field of $\Sigma$ in $S_{r}$ and $\{e_{\mu}\}_{\mu=1}^{n-1}$ is an orthonormal basis of $T_{p}\Sigma$.  If $\gamma(T)<0$ for any unit tangent vector $T\in T_{p}\Sigma$ and any $p\in \Sigma$, then we have that ${\rm Tr}_{\wedge^{2}\fg}q_{\omega^{\sharp}}<0$ for any non-trivial harmonic $1$-form $\omega$, and hence,  Proposition \ref{prop:meth2} implies the desired conclusion.

By using \eqref{eq:ber}, we compute 
\begin{align*}
\sum_{\mu=1}^{n-1}|B(T,e_{\mu})|^{2}&=
\sum_{\mu=1}^{n-1}\Big(a\langle T,e_{\mu}\rangle+b\sum_{i=1}^{\delta}\eta_{i}(T)\eta_{i}(e_{\mu})\Big)^{2}\\
&=\sum_{\mu=1}^{n-1}\Big(a^{2}\langle T,e_{\mu}\rangle^{2}+2ab\sum_{i=1}^{\delta}\langle T,e_{\mu}\rangle \eta_{i}(T)\eta_{i}(e_{\mu})+b^{2}\sum_{i=1}^{\delta}\eta_{i}(T)^{2}\eta_{i}(e_{\mu})^{2}\\
&\qquad +b^{2}\sum_{i\neq j}\eta_{i}(T)\eta_{i}(e_{\mu})\eta_{j}(T)\eta_{j}(e_{\mu})\Big{)} \\
&=a^{2}-2\sum_{i=1}^{\delta}\eta_{i}(T)^{2}+b^{2}\sum_{i=1}^{\delta}\eta_{i}(T)^{2}(1-\eta_{i}(\nu)^{2})
-b^{2}\sum_{i\neq j} \eta_{i}(T)\eta_{j}(T)\eta_{i}(\nu)\eta_{j}(\nu)\\
&=a^{2}+(b^{2}-2)\sum_{i=1}^{\delta}\eta_{i}(T)^{2}-b^{2}\Big(\sum_{i=1}^{\delta}\eta_{i}(T)\eta_{i}(\nu)\Big)^{2}.\\
\sum_{\mu=1}^{n-1}|B(\nu,e_{\mu})|^{2}&=\sum_{\mu=1}^{n-1}\Big(b\sum_{i=1}^{\delta}\eta_{i}(\nu)\eta_{i}(e_{\mu})\Big)^{2}\\
&=\sum_{\mu=1}^{n-1}b^{2}\Big(\sum_{i=1}^{\delta}\eta_{i}(\nu)^{2}\eta_{i}(e_{\mu})^{2}+ \sum_{i\neq j}\eta_{i}(\nu)\eta_{j}(\nu)\eta_{i}(e_{\mu})\eta_{j}(e_{\mu})\Big)\\
&=\sum_{i=1}^{\delta}b^{2}\eta_{i}(\nu)^{2}(1-\eta_{i}(\nu)^{2})-\sum_{i\neq j} b^{2}\eta_{i}(\nu)^{2}\eta_{j}(\nu)^{2}\\
&=b^{2}\sum_{i=1}^{\delta}\eta_{i}(\nu)^{2}-b^{2}\Big(\sum_{i=1}^{\delta}\eta_{i}(\nu)^{2}\Big)^{2}.
\end{align*}
where we used the relations $\sum_{\mu}\eta_{i}(e_{\mu})^{2}=1-\eta_{i}(\nu)^{2}$ and $0=\langle U_{i}, U_{j}\rangle=\sum_{\mu}\eta_{i}(e_{\mu})\eta_{j}(e_{\mu})+\eta_{i}(\nu)\eta_{j}(\nu)$ if $i\neq j$. 
Similarly, we have
\begin{align*}
\sum_{\mu=1}^{n-1}\langle B(T,T), B(e_{\mu},e_{\mu})\rangle
&=\sum_{\mu=1}^{n-1}\Big\{a|T|^{2}+b\sum_{i=1}^{\delta}\eta_{i}(T)^{2}\Big\}\Big\{a|e_{\mu}|^{2}+b\sum_{i=1}^{\delta}\eta_{i}(e_{\mu})^{2}\Big\}\\
&=\Big\{a+b\sum_{i=1}^{\delta}\eta_{i}(T)^{2}\Big\}\Big\{(n-1)a+\delta b-b\sum_{i=1}^{\delta}\eta_{i}(\nu)^{2}\Big\}\\
&=(n-1)a^{2}-\delta+\sum_{i=1}^{\delta}\eta_{i}(\nu)^{2}+\{-(n-1)+\delta b^{2}\}\sum_{i=1}^{\delta}\eta_{i}(T)^{2}
-b^{2}\Big(\sum_{i=1}^{\delta}\eta_{i}(T)^{2}\Big)\Big(\sum_{i=1}^{\delta}\eta_{i}(\nu)^{2}\Big).\\
\sum_{\mu=1}^{n-1}\langle B(\nu,\nu), B(e_{\mu},e_{\mu})\rangle
&=(n-1)a^{2}-\delta+\sum_{i=1}^{\delta}\eta_{i}(\nu)^{2}+\{-(n-1)+\delta b^{2}\}\sum_{i=1}^{\delta}\eta_{i}(\nu)^{2} -b^{2}\Big(\sum_{i=1}^{\delta}\eta_{i}(\nu)^{2}\Big)^{2}.
\end{align*}

We put
\[
P:=\sum_{i=1}^{\delta}\eta_{i}(T)^{2},\quad Q:=\sum_{i=1}^{\delta}\eta_{i}(\nu)^{2}.
\]
Note that for any unit tangent vector $X$,  we have
\begin{align*}
\sum_{i=1}^{\delta}\eta_{i}(X)^{2}=\sum_{i=1}^{\delta}\langle U_{i},X\rangle^{2}\leq \|X\|^{2}=1
\end{align*}
since $\{U_{1},\ldots, U_{\delta}\}$ is orthonormal. By the above computations, we see
\begin{align}\label{eq:gamm}
\gamma(T)&=2\Big{\{}a^{2}+(b^{2}-2)P-b^{2}\Big(\sum_{i=1}^{\delta}\eta_{i}(T)\eta_{i}(\nu)\Big)^{2}
+b^{2}Q-b^{2}Q^{2}\Big{\}}\\
&\quad -\Big{\{}(n-1)a^{2}-\delta+Q+\{-(n-1)+\delta b^{2}\}P-b^{2}PQ \nonumber\\
&\qquad +(n-1)a^{2}-\delta+Q+\{-(n-1)+\delta b^{2}\}Q-b^{2}Q^{2}\Big{\}}\nonumber\\
&=(-2n+4)a^{2}+\{(2-\delta) b^{2}+n-5\}P-2b^{2}\Big(\sum_{i=1}^{\delta}\eta_{i}(T)\eta_{i}(\nu)\Big)^{2}+b^{2}PQ\nonumber\\
&\quad +\{(2-\delta)b^{2}+n-3\}Q-b^{2}Q^{2}+2\delta.\nonumber
\end{align}


First, we consider the case when $\mathbb{K}=\mathbb{C}$. Since  the geodesic sphere $S_{r}^{n}\subset \mathbb{C}P^{m}$ is odd-dimensional, we may assume $n\geq 3$. Using the facts that $\delta=1$,  $0\leq P\leq 1$ and $0\leq  Q\leq 1$, we see
\begin{align*}
&-2b^{2}\Big(\sum_{i=1}^{\delta}\eta_{i}(T)\eta_{i}(\nu)\Big)^{2}+b^{2}PQ
+\{(2-\delta)b^{2}+n-3\}Q-b^{2}Q^{2}\\
&=-b^{2}PQ+(b^{2}+n-3)Q-b^{2}Q^{2}\\
&\leq (b^{2}+n-3)Q-b^{2}Q^{2}\leq 
\begin{cases}
n-3 & \textup{if $b^{2}+n-3\geq  2b^{2}$}\\
(b^{2}+n-3)^{2}/4b^{2} & \textup{if $b^{2}+n-3< 2b^{2}$}\\
\end{cases}.
\end{align*}
Suppose $n\geq 5$.
Then, if $b^{2}+n-3< 2b^{2}$ (or equivalently, $\tan^{2}r> n-3$),  it turns out  that by using \eqref{eq:gamm}$,\gamma(T)$ cannot be bounded from above by any negative constant.  Thus, we  may assume that $b^{2}+n-3\geq 2b^{2}$ or equivalently, $\tan^{2}r\leq n-3$. Then, by \eqref{eq:gamm}, we see
\begin{align*}
\gamma(T)&\leq (-2n+4)a^{2}+(b^{2}+n-5)+(n-3)+2\\
&=(-2n+4)a^{2}+b^{2}+2n-6.
\end{align*}
Thus, if furthermore
\[
(-2n+4)\cot^{2}r+\tan^{2}r+2n-6<0,
\]
then we obtain $\gamma(T)<0$. It is easy to see that this inequality satisfies when
\[
\tan^{2}r<\sqrt{n^{2}-4n+5}-(n-3).
\]
(Note that $\sqrt{n^{2}-4n+5}-(n-3)<n-3$ if $n\geq 5$). For example, when $n=5$, the last inequality is equivalent to that  $\tan ^{2}r<\sqrt{10}-2=1.162\cdots.$ Moreover, we remark that  $\sqrt{n^{2}-4n+5}-(n-3)>1$ and tends to $1$ if $n\to \infty$.

One can obtain a similar result even when $n=3$ by estimating $\gamma(T)$. However, it turns out that the result is not better than the ones given in Theorem \ref{thm:a1}. Thus, we omit the details for the case when $n=3$.

Next, we consider the case when $\mathbb{K}=\mathbb{H}$ or $\mathbb{O}$. Note that we may assume $n\geq 7$.  In this case, by \eqref{eq:gamm}, we have
\[
\gamma(T)\leq (-2n+4)a^{2}+(n-5)P+\{(3-\delta)b^{2}+n-3\}Q-b^{2}Q^{2}+2\delta,
\]
where we used the fact that $b^{2}PQ\leq b^{2}Q$ since $0\leq P\leq 1$ and $0\leq Q\leq 1$.

If $\delta=3$, we have
\[
\{(3-\delta)b^2+n-3\}Q-b^{2}Q^{2}=(n-3)Q-b^{2}Q^{2}\leq 
\begin{cases}
n-3-b^{2} & \textup{if $n-3\geq  2b^{2}$}\\
(n-3)^{2}/4b^{2} & \textup{if $n-3< 2b^{2}$}
\end{cases}
\]
since $0\leq Q\leq 1$. When $n-3< 2b^{2}$, it turns out that $\gamma(T)$ cannot be bounded from above by any negative constant, and hence, we may assume that $n-3\geq 2b^{2}$ or equivalently, $\tan^{2}r\leq (n-3)/2$. Then we see
\begin{align*}
\gamma(T)&\leq (-2n+4)a^{2}+(n-5)+(n-3-b^{2})+2\cdot 3\\
&=(-2n+4)a^{2}-b^{2}+2n-2.
\end{align*}
Using this,  we easily see that if furthermore
\[
\tan^{2}r<n-1-\sqrt{n^{2}-4n-5}
\]
then $\gamma(T)<0$. For example, if $n=7$, then $\tan^{2}r<6-\sqrt{16}=2$. Moreover, we remark that $\lim_{n\to \infty} (n-1-\sqrt{n^{2}-4n-5})=1$.

In the case when $\mathbb{K}=\mathbb{O}$, we have $\delta=7$ and $n=15$. Then,
\begin{align*}
\{(3-\delta)b^2+n-3\}Q-b^{2}Q^{2}
&=(-4b^{2}+12)Q-b^{2}Q^{2}.
\end{align*}
We may assume that $(-4b^{2}+12)/2b^{2}\geq1$ or equivalently $b^{2}\leq 2$ so that $(-4b^{2}+12)Q-b^{2}Q^{2}\leq (-4b^{2}+12)-b^{2}=-5b^{2}+12$. Otherwise, it turns out that $\gamma(T)$ cannot be bounded from above by any negative constant. When $b^{2}=\tan^{2}r\leq 2$, then 
\begin{align*}
\gamma(T)&\leq (-2\cdot 15+4)a^{2}+(15-5)+(-5b^{2}+12)+2\cdot 7\\
&=-26a^{2}-5b^{2}+36.
\end{align*}
Thus it follows that if 
\[
\tan^{2}r<\frac{1}{5}(18-\sqrt{194})=0.8143\cdots
\]
then $\gamma(T)<0$.

Summarizing the argument, we have proved that if 
\[
\tan^{2}r<
\begin{cases}
\sqrt{n^{2}-4n+5}-(n-3) & \textup{if $\mathbb{K}=\mathbb{C}$ and $n\geq 5$}\\
n-1-\sqrt{n^{2}-4n-5} & \textup{if $\mathbb{K}=\mathbb{H}$}\\
\frac{1}{5}(18-\sqrt{194})& \textup{if $\mathbb{K}=\mathbb{O}$}
\end{cases},
\]
then we have $\gamma(T)<0$ for any unit tangent vector $T\in T_{p}\Sigma$ and we obtain the required conclusion.
\end{proof}

\section{Proof of Theorem \ref{thm:main0}}\label{sec:pf}

The aim of this section is to show the following result (Theorem \ref{thm:main0}).

\begin{theorem}\label{thm:main}
Let $\Sigma$ be a closed minimal hypersurface (not necessarily two-sided)  in a compact semi-simple Riemannian symmetric space $M=G/K$. If $\Sigma$ is unstable (i.e. ${\rm ind}(\Sigma)\geq 1$), then we have
\begin{align}\label{eq:main}
{\rm index}(\Sigma)\geq \frac{2}{d(d-1)+2(2n-3)}b_{1}(\Sigma),
\end{align}
where  $d={\rm dim} G$ and $n={\rm dim}M$. 
\end{theorem}

The following proof is a refinement and an extention of the argument given by Ambrozio-Carlotto-Sharp \cite{ACS2} and Mendes-Radeschi  \cite{MR}. 
We first prove Theorem \ref{thm:main} in the case when $\Sigma$ is two-sided, and then, we shall deal with the one-sided case. 
Throughout this section,  we suppose that $M$ is a semi-simple compact Riemannian symmetric space, and by assuming $M=N$, we use the same notation given in the previous sections (e.g. $\nabla^{M}=\onab$, $R^{M}=\overline{R}$, etc).
We begin with the following formula. 

\begin{proposition}\label{prop:jac}
Let $M=N$ be a compact semi-simple Riemannian symmetric space and $\Sigma$ be a two-sided minimal hypersurface  in $M$. Fix arbitrary point $p\in \Sigma$. Then for any harmonic $1$-form $\omega$ on $\Sigma$ and any $X,Y\in \fn_{p}$, we have
\begin{align*}
\mathcal{J}(\varphi_{\omega^{\sharp},X\wedge Y})(p)=-2\langle A(\nabla_{(Y^{\dagger})^{\top}}\omega^{\sharp},(X^{\dagger})^{\top}),\nu\rangle
+2\langle A(\nabla_{(X^{\dagger})^{\top}}\omega^{\sharp},(Y^{\dagger})^{\top}),\nu\rangle,
\end{align*}
where $\top$ means the orthogonal projection to $T_{p}\Sigma$.
\end{proposition}
\begin{remark}
{\rm A similar formula was proved in \cite[Proposition 3]{ACS2} when $M$ is the flat torus. 
}
\end{remark}

To prove this, we shall compute $\mathcal{J}(\varphi_{\omega^{\sharp}, X\wedge Y})=\Delta\varphi_{\omega^{\sharp}, X\wedge Y}+(\overline{\rm Ric}(\nu,\nu)+|A|^{2})\varphi_{\omega^{\sharp}, X\wedge Y}$ at the fixed point $p$.  We denote $\varphi_{\omega^{\sharp},X\wedge Y}$ by $\varphi_{X\wedge Y}$ for simplicity.
We take the geodesic normal coordinate $\{\p_{\mu}\}_{\mu=1}^{n-1}$ of $\Sigma$ around $p\in \Sigma$. Then, for any $X,Y\in \fn_{p}$, we have $\onab X^{\dagger}(p)=\onab Y^{\dagger}(p)=0$ by Lemma \ref{lem:key1} and hence, we see
\begin{align}\label{eq:lap3}
\Delta \varphi_{X\wedge Y}(p)=\sum_{\mu=1}^{n-1}\langle \onab_{\p_{\mu}}\onab_{\p_{\mu}}(\nu\wedge \omega^{\sharp}), X^{\dagger}\wedge Y^{\dagger}\rangle(p)
+\sum_{\mu=1}^{n-1}\langle \nu\wedge \omega^{\sharp}, \onab_{\p_{\mu}}\onab_{\p_{\mu}}(X^{\dagger}\wedge Y^{\dagger})\rangle(p).
\end{align}
Here, we note that
\begin{align}\label{eq:lap21}
\onab_{\p_{\mu}}\onab_{\p_{\mu}}(\nu\wedge \omega^{\sharp})
=(\onab_{\p_{\mu}}\onab_{\p_{\mu}}\nu)\wedge \omega^{\sharp}+2\onab_{\p_{\mu}}\nu\wedge \onab_{\p_{\mu}}\omega^{\sharp}+\nu\wedge (\onab_{\p_{\mu}}\onab_{\p_{\mu}}\omega^{\sharp}).
\end{align}
Recall that any compact semi-simple RSS $M$ is an Einstein manifold with positive Ricci curvature. 

\begin{lemma}
Suppose $M=N$ is a compact semi-simple RSS with $\overline{\rm Ric}=cg$ and $\omega$ is a harmonic $1$-form on $\Sigma$. Then, we have
\begin{align}
\label{eq:la11}
\sum_{\mu=1}^{n-1}\langle (\onab_{\p_{\mu}}\onab_{\p_{\mu}}\nu)\wedge \omega^{\sharp}, X^{\dagger}\wedge Y^{\dagger}\rangle(p)
&= -|A|^{2}\varphi_{X\wedge Y}.\\
\label{eq:la12}
\sum_{\mu=1}^{n-1}2\langle\onab_{\p_{\mu}}\nu\wedge \onab_{\p_{\mu}}\omega^{\sharp},X^{\dagger}\wedge Y^{\dagger}\rangle(p)
&=-2\langle A(\nabla_{(Y^{\dagger})^{\top}}\omega^{\sharp},(X^{\dagger})^{\top}),\nu\rangle-2\sum_{\mu=1}^{n-1}\langle A(\p_{\mu}, (X^{\dagger})^{\top}),\nu\rangle \langle A(\p_{\mu},\omega^{\sharp}),Y^{\dagger}\rangle\\
&\quad +2\langle A(\nabla_{(X^{\dagger})^{\top}}\omega^{\sharp},(Y^{\dagger})^{\top}),\nu\rangle+2\sum_{\mu=1}^{n-1}\langle A(\p_{\mu}, (Y^{\dagger})^{\top}),\nu\rangle \langle A(\p_{\mu},\omega^{\sharp}),X^{\dagger}\rangle. \nonumber\\
\label{eq:la13}
\sum_{\mu=1}^{n-1}\langle\nu\wedge (\onab_{\p_{\mu}}\onab_{\p_{\mu}}\omega^{\sharp}),X^{\dagger}\wedge Y^{\dagger}\rangle(p)
&= -\langle \overline{R}(Y^{\dagger}, \nu)\nu, \omega^{\sharp}\rangle\langle X^{\dagger}, \nu\rangle-2\sum_{\mu=1}^{n-1}\langle A(\p_{\mu}, (Y^{\dagger})^{\top}),\nu\rangle \langle A(\p_{\mu},\omega^{\sharp}),X^{\dagger}\rangle\\
&\quad+ \langle \overline{R}(X^{\dagger}, \nu)\nu, \omega^{\sharp}\rangle\langle Y^{\dagger}, \nu\rangle+2\sum_{\mu=1}^{n-1}\langle A(\p_{\mu}, (X^{\dagger})^{\top}),\nu\rangle \langle A(\p_{\mu},\omega^{\sharp}),Y^{\dagger}\rangle\nonumber\\
&\quad +c\cdot \varphi_{X\wedge Y}.\nonumber\\
\label{eq:la14}
\sum_{\mu=1}^{n-1}\langle \nu\wedge \omega^{\sharp}, \onab_{\p_{\mu}}\onab_{\p_{\mu}}(X^{\dagger}\wedge Y^{\dagger})\rangle(p)
&=-\langle \overline{R}(X^{\dagger}, \nu)\nu, \omega^{\sharp}\rangle\langle Y^{\dagger}, \nu\rangle\\
&\quad+\langle \overline{R}(Y^{\dagger}, \nu)\nu, \omega^{\sharp}\rangle\langle X^{\dagger}, \nu\rangle-2c\cdot \varphi_{X\wedge Y}.\nonumber
\end{align}
\end{lemma}

\begin{proof}
First, we shall show \eqref{eq:la11}. We have
\begin{align*}
\langle \onab_{\p_{\mu}}\onab_{\p_{\mu}}\nu, \p_{\xi}\rangle
&=\p_{\mu}\langle \onab_{\p_{\mu}}\nu, \p_{\xi}\rangle-\langle\onab_{\p_{\mu}}\nu,\onab_{\p_{\mu}}\p_{\xi}\rangle\\
&=-\p_{\mu}\langle A(\p_{\mu},\p_{\xi}), \nu\rangle-\langle\onab_{\p_{\mu}}\nu,\onab_{\p_{\mu}}\p_{\xi}\rangle\\
&=-\langle \onab_{\p_{\mu}}(A(\p_{\mu},\p_{\xi})),\nu\rangle-\langle A(\p_{\mu},\p_{\xi}),\onab_{\p_{\mu}}\nu\rangle-\langle\onab_{\p_{\mu}}\nu,\onab_{\p_{\mu}}\p_{\xi}\rangle.
\end{align*}
Since $A(\p_{\mu},\p_{\xi})$ is parallel to $\nu$ and $\langle \onab_{\p_{\mu}}\nu,\nu\rangle=0$ along $\Sigma$, we have $\langle A(\p_{\mu},\p_{\xi}),\onab_{\p_{\mu}}\nu\rangle=0$. Moreover, we have $\onab_{\p_{\mu}}\p_{\xi}(p)=A(\p_{\mu},\p_{\xi})(p)$ at the fixed point $p\in \Sigma$, and hence, $\langle\onab_{\p_{\mu}}\nu,\onab_{\p_{\mu}}\p_{\xi}\rangle(p)=0$. Thus, we see
\begin{align}\label{eq:la112}
\langle \onab_{\p_{\mu}}\onab_{\p_{\mu}}\nu, \p_{\xi}\rangle(p)
&=-\langle \onab_{\p_{\mu}}(A(\p_{\mu},\p_{\xi})),\nu\rangle\\
&=-\langle (\nabla_{\p_{\mu}}^{\perp}A)(\p_{\mu},\p_{\xi}),\nu\rangle\nonumber\\
&=-\langle (\nabla_{\p_{\mu}}^{\perp}A)(\p_{\xi},\p_{\mu}),\nu\rangle\nonumber\\
&=-\langle (\nabla_{\p_{\xi}}^{\perp}A)(\p_{\mu},\p_{\mu}),\nu\rangle-\langle(\overline{R}(\p_{\mu},\p_{\xi})\p_{\mu})^{\perp},\nu\rangle \nonumber\\
&=-\langle (\nabla_{\p_{\xi}}^{\perp}A)(\p_{\mu},\p_{\mu}),\nu\rangle+\langle \overline{R}(\p_{\xi},\p_{\mu})\p_{\mu},\nu\rangle,\nonumber
\end{align}
where $\nabla^{\perp}$ is the normal connection and $(\nabla_{X}^{\perp}A)(Y,Z)=\nabla^{\perp}_{X}(A(Y,Z))-A(\nabla_{X}Y,Z)-A(Y,\nabla_{X}Z)$, and we used the Codazzi equation in the forth equality. Since $\Sigma$ is a minimal hypersurface, we have 
\[
\sum_{\mu,\lambda=1}^{n-1}g^{\mu\lambda}\langle A(\p_{\mu},\p_{\lambda}),\nu\rangle =0
\]
around the point $p$. Differentiating this equation in the direction of $\p_{\xi}$, we see 
\begin{align*}
0&=\p_{\xi}\Big(\sum_{\mu,\lambda=1}^{n-1}g^{\mu\lambda}\langle A(\p_{\mu},\p_{\lambda}),\nu\rangle\Big)(p)=\sum_{\mu=1}^{n-1}\partial_{\xi}\langle A(\p_{\mu},\p_{\mu}),\nu\rangle(p)\\
&=\sum_{\mu=1}^{n-1} \langle \nabla_{\p_{\xi}}^{\perp}(A(\p_{\mu},\p_{\mu})),\nu\rangle(p)=\sum_{\mu=1}^{n-1} \langle (\nabla_{\p_{\xi}}^{\perp}A)(\p_{\mu},\p_{\mu})),\nu\rangle(p)
\end{align*}
for any $\xi=1,\ldots,n-1$.
Moreover, since $M$ is Einstein, we have
\[
\sum_{\mu=1}^{n-1}\langle \overline{R}(\p_{\xi},\p_{\mu})\p_{\mu},\nu\rangle={\overline{\rm Ric}}(\p_{\xi},\nu)=c\langle \p_{\xi},\nu\rangle=0
\]
for any $\xi=1,\ldots,n-1$. Therefore, by \eqref{eq:la112}, we obtain
\[
\sum_{\mu=1}^{n-1}\langle \onab_{\p_{\mu}}\onab_{\p_{\mu}}\nu, \p_{\xi}\rangle(p)=0\quad \textup{for any $\xi=1,\ldots, n-1$.}
\]
In particular, $\sum_{\mu=1}^{n-1}\onab_{\p_{\mu}}\onab_{\p_{\mu}}\nu (p)$ is proportional to $\nu$.
Since $\langle \onab_{\p_{\mu}}\nu,\nu\rangle=0$ along $\Sigma$, we see
\begin{align*}
\sum_{\mu=1}^{n-1}\langle \onab_{\p_{\mu}}\onab_{\p_{\mu}}\nu, \nu\rangle(p)=-\sum_{\mu=1}^{n-1}|\onab_{\p_{\mu}}\nu|^{2}=-|A|^{2}
\end{align*}
and this implies \eqref{eq:la11}.

Next, we shall prove  \eqref{eq:la12}. Note that 
\begin{align}\label{eq:la121}
2\sum_{\mu=1}^{n-1}\langle\onab_{\p_{\mu}}\nu\wedge \onab_{\p_{\mu}}\omega^{\sharp},X^{\dagger}\wedge Y^{\dagger}\rangle=2\sum_{\mu=1}^{n-1}\Big(\langle \onab_{\p_{\mu}}\nu, X^{\dagger}\rangle\langle\onab_{\p_{\mu}}\omega^{\sharp},Y^{\dagger}\rangle-\langle \onab_{\p_{\mu}}\nu, Y^{\dagger}\rangle\langle\onab_{\p_{\mu}}\omega^{\sharp},X^{\dagger}\rangle\Big).
\end{align}
Since $\omega$ is a closed $1$-form, we have $\langle \nabla_{T}\omega^{\sharp},U\rangle=\langle \nabla_{U}\omega^{\sharp},T\rangle$ for any $T,U\in T_{p}\Sigma$.
Thus, for any $X\in T_{p}M$, we see
\begin{align*}
\langle \onab_{\p_{\mu}}\omega^{\sharp}, X\rangle
&=\langle \nabla_{\p_{\mu}}\omega^{\sharp}, X^{\top}\rangle+\langle \onab_{\p_{\mu}}\omega^{\sharp}, \nu\rangle\langle \nu,X\rangle\\
&=\langle \nabla_{X^{\top}}\omega^{\sharp}, \p_{\mu}\rangle+ \langle A(\p_{\mu},\omega^{\sharp}), X\rangle.
\end{align*}
Therefore, we have
\begin{align*}
\sum_{\mu=1}^{n-1}\langle \onab_{\p_{\mu}}\nu, X^{\dagger}\rangle\langle\onab_{\p_{\mu}}\omega^{\sharp},Y^{\dagger}\rangle(p)
&=-\sum_{\mu=1}^{n-1}\langle A(\p_{\mu}, (X^{\dagger})^{\top}),\nu\rangle\{\langle \nabla_{(Y^{\dagger})^{\top}}\omega^{\sharp}, \p_{\mu}\rangle+ \langle A(\p_{\mu},\omega^{\sharp}), Y^{\dagger}\rangle\}(p)\\
&=-\langle A(\nabla_{(Y^{\dagger})^{\top}}\omega^{\sharp},(X^{\dagger})^{\top}),\nu\rangle
-\sum_{\mu=1}^{n-1}\langle A(\p_{\mu}, (X^{\dagger})^{\top}),\nu\rangle \langle A(\p_{\mu},\omega^{\sharp}),Y^{\dagger}\rangle.
\end{align*}
Substituting this to \eqref{eq:la121}, we obtain \eqref{eq:la12}.

We shall consider  \eqref{eq:la13}. We have
\begin{align*}
\langle\onab_{\p_{\mu}}\onab_{\p_{\mu}}\omega^{\sharp},\p_{\xi}\rangle(p)
&=\p_{\mu}\langle\onab_{\p_{\mu}}\omega^{\sharp},\p_{\xi}\rangle-\langle \onab_{\p_{\mu}}\omega^{\sharp},\onab_{\p_{\mu}}\p_{\xi}\rangle\\
&=\p_{\mu}\langle\nabla_{\p_{\mu}}\omega^{\sharp},\p_{\xi}\rangle-\langle \onab_{\p_{\mu}}\omega^{\sharp},A(\p_{\mu},\p_{\xi})\rangle\\
&=\langle \nabla_{\p_{\mu}}\nabla_{\p_{\mu}}\omega^{\sharp},\p_{\xi}\rangle-\langle A(\p_{\mu},\omega^{\sharp}),A(\p_{\mu},\p_{\xi})\rangle.
\end{align*}
Thus, 
\begin{align}\label{eq:la132}
\langle\nu\wedge \onab_{\p_{\mu}}\onab_{\p_{\mu}}\omega^{\sharp},X^{\dagger}\wedge Y^{\dagger}\rangle(p)
&=\sum_{\xi=1}^{n-1}
\langle\onab_{\p_{\mu}}\onab_{\p_{\mu}}\omega^{\sharp},\p_{\xi}\rangle\langle\nu\wedge \p_{\xi},X^{\dagger}\wedge Y^{\dagger}\rangle\\
&=\sum_{\xi=1}^{n-1}
\langle\onab_{\p_{\mu}}\onab_{\p_{\mu}}\omega^{\sharp},\p_{\xi}\rangle\Big(\langle\nu,X^{\dagger}\rangle\langle \p_{\xi}, Y^{\dagger}\rangle-\langle\nu,Y^{\dagger}\rangle\langle \p_{\xi}, X^{\dagger}\rangle\Big)\nonumber\\
&=\Big(\langle \nabla_{\p_{\mu}}\nabla_{\p_{\mu}}\omega^{\sharp},(Y^{\dagger})^{\top}\rangle-\langle A(\p_{\mu},\omega^{\sharp}),A(\p_{\mu},(Y^{\dagger})^{\top})\rangle\Big)\langle  X^{\dagger},\nu\rangle \nonumber\\
&\quad -\Big(\langle \nabla_{\p_{\mu}}\nabla_{\p_{\mu}}\omega^{\sharp},(X^{\dagger})^{\top}\rangle-\langle A(\p_{\mu},\omega^{\sharp}),A(\p_{\mu},(X^{\dagger})^{\top})\rangle\Big)\langle Y^{\dagger},\nu\rangle. \nonumber
\end{align}
Note that we have
\begin{align}\label{eq:la133}
\langle A(\p_{\mu},\omega^{\sharp}),A(\p_{\mu},(Y^{\dagger})^{\top})\rangle\langle X^{\dagger},\nu\rangle=\langle A(\p_{\mu},(Y^{\dagger})^{\top}), \nu\rangle\langle A(\p_{\mu},\omega^{\sharp}),X^{\dagger}\rangle.
\end{align}

Since $\omega$ is a harmonic $1$-form, the Bochner-Weitzenb\"ock formula and the Gauss equation for minimal hypersurface $\Sigma\to N$  shows that, for any $T\in T_{p}\Sigma$, we have
\begin{align*}
\sum_{\mu=1}^{n-1}\langle \nabla_{\p_{\mu}}\nabla_{\p_{\mu}}\omega^{\sharp},T\rangle(p)
&=\sum_{\mu=1}^{n-1}\langle R(\omega^{\sharp},\p_{\mu})\p_{\mu}, T\rangle\\
&=\sum_{\mu=1}^{n-1}\langle \overline{R}(\omega^{\sharp},\p_{\mu})\p_{\mu}, T\rangle-\sum_{\mu=1}^{n-1}\langle A(\p_{\mu},\omega^{\sharp}), A(\p_{\mu}, T)\rangle\\
&={\overline{\rm Ric}}(\omega^{\sharp}, T)-\langle\overline{R}(\omega^{\sharp},\nu)\nu, T\rangle-\sum_{\mu=1}^{n-1}\langle A(\p_{\mu},\omega^{\sharp}), A(\p_{\mu}, T)\rangle.
\end{align*}
Substituting this formula to \eqref{eq:la132} and using \eqref{eq:la133}, we obtain \eqref{eq:la13}.

Finally, we shall show \eqref{eq:la14}. The equation \eqref{eq:nn} shows that 
\begin{align*}
\langle \nu\wedge \omega^{\sharp}, \onab_{\p_{\mu}}\onab_{\p_{\mu}}(X^{\dagger}\wedge Y^{\dagger})\rangle(p)
&=\langle \nu\wedge \omega^{\sharp}, -(\overline{R}(X^{\dagger},\p_{\mu})\p_{\mu})\wedge Y^{\dagger}+(\overline{R}(Y^{\dagger},\p_{\mu})\p_{\mu})\wedge X^{\dagger}\rangle(p),
\end{align*}
where
\begin{align*}
\sum_{\mu=1}^{n-1}\langle \nu\wedge \omega^{\sharp}, (\overline{R}(X^{\dagger},\p_{\mu})\p_{\mu})\wedge Y^{\dagger}\rangle
&=\sum_{\mu=1}^{n-1}\Big(\langle \overline{R}(X^{\dagger},\p_{\mu})\p_{\mu},\nu\rangle\langle Y^{\dagger}, \omega^{\sharp}\rangle
-\langle \overline{R}(X^{\dagger},\p_{\mu})\p_{\mu},\omega^{\sharp}\rangle\langle Y^{\dagger}, \nu\rangle\Big)\\
&=\overline{\rm Ric}(X^{\dagger},\nu)\langle Y^{\dagger}, \omega^{\sharp}\rangle
-\overline{\rm Ric}(X^{\dagger},\omega^{\sharp})\langle Y^{\dagger},\nu\rangle+\langle \overline{R}(X^{\dagger}, \nu)\nu, \omega^{\sharp}\rangle\langle Y^{\dagger}, \nu\rangle\\
&=c\varphi_{X\wedge Y}
+\langle \overline{R}(X^{\dagger}, \nu)\nu, \omega^{\sharp}\rangle\langle Y^{\dagger}, \nu\rangle.
\end{align*} 
This implies \eqref{eq:la14}.
\end{proof}

By \eqref{eq:lap3}--\eqref{eq:la14}, we obtain
\begin{align*}
\mathcal{J}(\varphi_{X\wedge Y})(p)&=\Delta\varphi_{X\wedge Y}+(c+|A|^{2})\varphi_{X\wedge Y}\\
&=
-2\langle A(\nabla_{(Y^{\dagger})^{\top}}\omega^{\sharp},(X^{\dagger})^{\top}),\nu\rangle
+2\langle A(\nabla_{(X^{\dagger})^{\top}}\omega^{\sharp},(Y^{\dagger})^{\top}),\nu\rangle
\end{align*}
for any $X,Y\in \fn_{p}$, and this proves Proposition \ref{prop:jac}.

As a consequence,  if a harmonic $1$-form $\omega$ satisfies $\mathcal{J}(\varphi_{\omega^{\sharp},\alpha})=0$ along $\Sigma$ for any $\alpha\in \wedge^{2}\fg$, then we obtain 
\[A(\nabla_{X}\omega^{\sharp},Y)=A(\nabla_{Y}\omega^{\sharp},X)\]
 for any $X,Y\in T_{p}\Sigma$ and any $p\in \Sigma$ since $\fn_{p}$ is isomorphic to $T_{p}M$ by the correspondence $X\mapsto X^{\dagger}(p)$. Thus, we define  a linear subspace of $\mathcal{H}$ by
 \begin{align*}
\mathcal{H}_{1}:=\{\omega\in \mathcal{H}\mid A(\nabla_{X}\omega^{\sharp},Y)=A(\nabla_{Y}\omega^{\sharp},X),\,\forall X,Y\in T_{p}\Sigma, \forall p\in \Sigma\}.
 \end{align*}
 Then, $\mathcal{H}_{0}=\{\omega\in \mathcal{H}\mid \mathcal{J}(\varphi_{\omega^{\sharp},\alpha})=0, \forall \alpha\in \wedge^{2}\fg\}$ is a linear subspace of $\mathcal{H}_{1}$.
 
We shall estimate the dimension of $\mathcal{H}_{1}$.
The following proposition is a generalization of \cite[Proposition 5]{ACS2}, and the proof is inspired by the argument of Li \cite{Li}.

\begin{proposition}\label{prop:key4}
Suppose $\Sigma^{n-1}$ is a closed two-sided minimal hypersurface in a compact semi-simple symmetric space $M^{n}$.  Then we have
\[
{\rm dim}\mathcal{H}_{1}\leq 2n-3.
\]
\end{proposition}
\begin{remark}
{\rm 
This proposition was proved in \cite{ACS2} when $M$ is a flat torus and under an additional assumption that ``there is a point $p\in \Sigma$ so that the eigenvalues of shape operator $S_{\nu}(p)$ at $p$ are distinct''. Our proof  below is a refinement of the argument  in \cite{ACS2, Li} combining with a technical result given in Appendix \ref{A:smooth}, and it turns out that we can remove the additional assumption even when $M$ is a flat torus. 

We also remark that in \cite{MR}, Mendes-Radeschi provide more precise estimate of ${\rm dim}\mathcal{H}_{0}(\leq {\rm dim}\mathcal{H}_{1})$ when there is a point $p$ such that $S_{\nu}(p)$ has distinct eigenvalues. 
}
\end{remark}
 
To prove this, we use the following fact.
 
 \begin{lemma}\label{lem:aeq}
Let $\tau$ be a closed $1$-form on $\Sigma$, and we set $T=\tau^{\sharp}$. Then, 
$A(\nabla_XT, Y)=A(\nabla_YT, X)$ for any $X,Y\in T_p\Sigma$ if and only if $ [S_\nu, \nabla T]=0$,
where $S_{\nu}: T_{p}\Sigma\to T_{p}\Sigma$ is the shape operator at $p$ defined by $S_{\nu}(X):=-(\onab_{X}\nu)^{\top}$. 

\end{lemma}
\begin{proof}
Recall that, if $\tau$ is a closed $1$-form, then
$
\langle \nabla_XT, Y\rangle =\langle \nabla_YT, X\rangle
$
for any $X,Y\in T_p\Sigma$.
A direct computation shows that 
\begin{align*}
\langle [S_\nu, \nabla T]\p_\mu, \p_\xi\rangle
&=\langle S_\nu(\nabla_{\p_\mu}T), \p_\xi\rangle-\langle \nabla_{S_\nu(\p_\mu)}T, \p_\xi\rangle\\
&=\langle  A(\p_\xi,\nabla_{\p_\mu}T),\nu\rangle-\langle \nabla_{\p_\xi}T, S_\nu(\p_\mu)\rangle\\
&=\langle \onab_{\p_\xi}\nabla_{\p_\mu}T,\nu\rangle+\langle \nabla_{\p_\xi}T, \onab_{\p_\mu}\nu\rangle\\
&=\langle \onab_{\p_\xi}\nabla_{\p_\mu}T,\nu\rangle-\langle \onab_{\p_\mu}\nabla_{\p_\xi}T, \nu\rangle\\
&=\langle \overline{R}(\p_\xi,\p_\mu)T,\nu\rangle-\Big\langle\onab_{\p_\xi}(A(\p_\mu,T))-\onab_{\p_\mu}(A(\p_\xi,T)),\nu\Big\rangle.
\end{align*}
Moreover, at the point $p$, the Codazzi equation shows that 
\begin{align*}
\Big\langle \onab_{\p_\xi}(A(\p_\mu,T))-\onab_{\p_\mu}(A(\p_\xi,T)),\nu\Big\rangle(p)
&=\Big\langle (\nabla^\perp_{\p_\xi}A)(\p_\mu,T)-(\nabla^\perp_{\p_\mu}A)(\p_\xi,T),\nu\Big\rangle\\
&\quad +\Big\langle A(\p_\mu, \nabla_{\p_\xi}T)-A(\p_\xi, \nabla_{\p_\mu}T),\nu\Big\rangle\\
&=\langle \overline{R}(\p_\xi,\p_\mu)T,\nu\rangle+\Big\langle A(\p_\mu, \nabla_{\p_\xi}T)-A(\p_\xi, \nabla_{\p_\mu}T),\nu\Big\rangle.
\end{align*}
Substituting this to the previous equation, we obtain
\[
\langle [S_\nu, \nabla T]\p_\mu, \p_\xi\rangle=-\Big\langle A(\p_\mu, \nabla_{\p_\xi}T)-A(\p_\xi, \nabla_{\p_\mu}T),\nu\Big\rangle
\]
for any $\mu,\xi=1,\ldots, n-1$.
This implies the conclusion.
\end{proof}

Now, we give a proof of Proposition \ref{prop:key4}.

\begin{proof}[Proof of Proposition \ref{prop:key4}]
  By Lemma \ref{lem:aeq}, $\mathcal{H}_1$ can be written as 
  \[\mathcal{H}_1=\{\omega\in \mathcal{H}\mid [S_{\nu}, \nabla\omega^\sharp]=0\}. \]
  Notice that, for any harmonic $1$-form $\omega$, $\nabla \omega^{\sharp}(X,Y)=\langle \nabla_X\omega^{\sharp}, Y\rangle$ is symmetric and hence $\nabla \omega^\sharp: T_p\Sigma\to T_p\Sigma$ is diagonalizable and has real eigenvalues. Therefore, if $\omega\in \mathcal{H}_{1}$,  by using Proposition \ref{prop:key5},  there exists a  geodesic ball $B_\rho(p)\subset \Sigma$ around some point $p\in \Sigma$ and  a smooth orthonormal frame $\{E_{\mu}\}_{\mu=1}^{n-1}$ on $B_{\rho}(p)$ which simultaneously diagonalizes $S_{\nu}$ and $\nabla{\omega^{\sharp}}$.  In particular, we have
  \begin{align}
    \nabla\omega^\sharp(E_\mu, E_\xi)=0 \quad \text{for} \quad \mu\neq \xi. \label{eq:diagonal}
  \end{align}
  Note that the proof of \cite[Proposition 5]{ACS2} essentially use the relation \eqref{eq:diagonal} and smoothness of principal vectors, both of them are derived by assuming that there is a point $p\in \Sigma$ where all the principal curvatures are distinct. Thanks to Lemma \ref{lem:aeq} and Proposition \ref{prop:key5}, we can remove the assumption of principal curvatures considered in \cite{ACS2}. The rest of our proof completely follows the argument in \cite{ACS2}. 
  
  For a harmonic 1-form $\omega$, $\nabla\omega^\sharp$ is trace-free, i.e., 
  \[\sum_{\mu=1}^{n-1}\nabla\omega^\sharp(E_\mu, E_\mu)=0. \]
  Combining this fact and \eqref{eq:diagonal}, $\nabla\omega^\sharp$ on $B_\rho(p)$ is completely determined by the functions 
  \[\nabla\omega^\sharp(E_1, E_1), \dots, \nabla\omega^\sharp(E_{n-2}, E_{n-2}). \] 
  Now, we consider $2n-3$ functions on $B_\rho(p)$ by 
  \begin{align*}
    \phi_{\omega,i}(q):=
    \begin{cases}
      \omega(E_i)(q) &\quad \text{for} \quad 1\leq i\leq n-1\\
      \nabla\omega^\sharp(E_{i-n+1}, E_{i-n+1})(q) &\quad \text{for} \quad n\leq i\leq 2n-3. 
    \end{cases}
  \end{align*} 
  Given any $q\in B_\rho(p)$, let $\gamma(t):[0, \tau]\to \Sigma^{n-1}$ be the unique geodesic connecting $p$ to $q$. Restricting the functions $\phi_{\omega,1}, \dots, \phi_{\omega, 2n-3}$ to the geodesic $\gamma(t)$, we obtain the functions 
  \[f_i(t):=\phi_{\omega,i}(\gamma(t)), \quad 1\leq i\leq 2n-3.\]
 Owing to the local smoothness of $\{E_{\mu}\}_{\mu=1}^{n-1}$, the same computation given in the proof of  \cite[Proposition 5]{ACS2} holds, and moreover, we see that $\{f_i\}$ solves a linear ODE system of normal form. In particular, the value along $\gamma$ is determined by the initial value
$(f_1(0), \dots, f_{2n-3}(0))=(\phi_{\omega, 1}(p), \dots, \phi_{\omega, 2n-3}(p))$, namely,
 the values $(\phi_{\omega,1}(q), \dots, \phi_{\omega, 2n-3}(q))$ $(q\in B_{\rho}(p))$ are uniquely determined by $(\phi_{\omega, 1}(p), \dots, \phi_{\omega, 2n-3}(p))$. 
  
  Let $\omega$ and $\omega'$ be two harmonic forms belonging to the subspace $\mathcal{H}_{1}\subset \mathcal{H}$.
If $\phi_{\omega,i}(p)=\phi_{\omega',i}(p)$ for any $i=1,\ldots, 2n-3$, then the above argument shows that $\omega(q)=\omega'(q)$ for any $q\in B_{\rho}(p)$, and by using the unique continuation theorem for harmonic forms \cite{AKS}, we obtain $\omega=\omega'$ on $\Sigma$.  Therefore  an element $\omega\in \mathcal{H}_{1}$ is  determined by $(\phi_{\omega, 1}(p), \dots, \phi_{\omega, 2n-3}(p))$, and this implies that the dimension of $\mathcal{H}_{1}$ is at most $2n-3$ as required.
\end{proof}

Finally, we give a proof of our main result (Theorem \ref{thm:main}).

\begin{proof}[Proof of Theorem \ref{thm:main}]
We first assume that $\Sigma^{n-1}$ is two-sided.  By proposition \ref{prop:key2}, we have ${\rm Tr}_{\wedge^{2}\fg}q_{\omega^{\sharp}}=0$ for any harmonic $1$-form $\omega$ on $\Sigma$. Since $\langle,\rangle_{\wedge^{2}\fg}$ is positive-definite on $\wedge^{2}\fg$, Proposition \ref{prop:ind}-(ii) shows an affine bound
${\rm index}(\Sigma)\geq C(b_{1}(\Sigma)-{\rm dim}\mathcal{H}_{0})$,
where we put $C=2/(d(d-1))$, $d={\rm dim}\fg$. Moreover, by Proposition \ref{prop:key4}, we have ${\rm dim}\mathcal{H}_{0}\leq {\rm dim}\mathcal{H}_{1}\leq 2n-3$ since $\mathcal{H}_{0}$ is a linear subspace of $\mathcal{H}_{1}$. Combining this, we obtain 
\begin{align}\label{eq:ab}
{\rm index}(\Sigma)\geq C(b_{1}(\Sigma)-D),
\end{align}
where we set $D=2n-3$.
If a real number $a$ satisfies $a\geq 1$, then we have $a\geq (a+c)/(1+c)$ for any $c\geq 0$. Therefore, if $\Sigma$ is unstable (i.e. ${\rm index}(\Sigma)\geq 1$), then \eqref{eq:ab} implies
\begin{align}\label{eq:ab2}
{\rm index}(\Sigma)\geq \frac{{\rm index}(\Sigma)+CD}{1+CD}\geq \frac{Cb_{1}(\Sigma)}{1+CD}=\frac{2}{d(d-1)+2(2n-3)}b_{1}(\Sigma).
\end{align}
This proves the required inequality.

Next, we consider the case when $\Sigma$ is one-sided. We set   $\widetilde{\Sigma}:=\{(p,\nu_{p})\mid p\in \Sigma, \textup{$\nu_{p}\in \nu_{p}\Sigma, |\nu_{p}|=1$}\}$. We endow $\widetilde{\Sigma}$ a manifold structure so that  a natural projection $\pi:\widetilde{\Sigma}\to \Sigma$ becomes a two-fold covering map. Then we obtain an immersion $\widetilde{f}:=f\circ \pi:\widetilde{\Sigma}\to M$, where $f: \Sigma\to M$ is the isometric immersion of $\Sigma$. We induce a Riemannian metric on $\widetilde{\Sigma}$  via the map $\widetilde{f}$ so that $\widetilde{f}: \widetilde{\Sigma}\to M$ becomes an isometric immersion. 

We put $\widetilde{\nu}_{(p,\nu_{p})}:=\nu_{p}$. Then, $\widetilde{\nu}$ defines a smooth normal vector field along $\widetilde{f}$, and in particular, the normal bundle $\nu\widetilde{\Sigma}$ of $\widetilde{f}$ is trivial.  
Note that, for the non-trivial deck transformation $\tau: \widetilde{\Sigma}\to \widetilde{\Sigma}$, we have $\widetilde{\nu}\circ \tau=-\widetilde{\nu}$ under natural identifications $(\widetilde{f}^{*}TM)_{\tau(\widetilde{p})}\simeq (\widetilde{f}^{*}TM)_{\widetilde{p}}\simeq ({f}^{*}TM)_{p}$ for $\widetilde{p}\in \widetilde{\Sigma}$ and $p=\pi(\widetilde{p})\in \Sigma$.

For any $V\in \Gamma(f^{*}TM)$, we define a natural lift $\widetilde{V}\in \Gamma(\widetilde{f}^{*}TM)$  by $\widetilde{V}(\widetilde{p}):=V(\pi(\widetilde{p}))=V(p)$. If $V$ is a normal vector field, we can associate a function $\varphi_{V}\in C^{\infty}(\widetilde{\Sigma})$ satisfying that  $\widetilde{V}=\varphi_{V}\cdot \widetilde{\nu}$. Then, since $\widetilde{\nu}\circ \tau=-\widetilde{\nu}$ and $\widetilde{V}\circ \tau=\widetilde{V}$, it holds that $\varphi_{V}\circ \tau=-\varphi_{V}$, i.e. $\varphi_{V}$ is an odd function with respect to the deck transformation.  Conversely, for any odd function $\varphi\in C^{\infty}(\widetilde{\Sigma})$, we obtain a well-defined smooth normal vector field $V$ along $f$ by $V(p)=\varphi(\widetilde{p})\cdot \widetilde{\nu}(\widetilde{p})$, where $\widetilde{p}\in \pi^{-1}(p)$.

Let $T$ be any tangent vector field along $f$ and $\widetilde{T}$ its natural lift. Since $\widetilde{T}$ is invariant under the deck transformation,  the function $\varphi_{\widetilde{T},w}=\langle \widetilde{\nu}\wedge \widetilde{T}, \Pi(w)\rangle$ is an odd function  for any $w\in \wedge^{2}\fg$, and it yields a normal vector field $V_{\widetilde{T},w}$ along $f$.  Thus, we obtain a map
\[
\Phi: \Gamma(T\Sigma)\times \wedge^{2}\fg\to \Gamma(\nu\Sigma),\quad \Phi(T,w)=V_{\widetilde{T},w}
\]
and a quadratic form 
\[
q_{T}(v,w):=Q(V_{\widetilde{T},v}, V_{\widetilde{T},w}).
\]
Note that, since $\pi: \widetilde{\Sigma}\to \Sigma$ is a two-fold covering (and $\pi$ is a Riemannian submersion), we have $Q_{\widetilde{\Sigma}}(\widetilde{V},\widetilde{V})=2Q(V, V)$, where $Q_{\widetilde{\Sigma}}$ is the index form of $\widetilde{\Sigma}$.

Now, we take a harmonic $1$-form $\omega$ on $\Sigma$. Then, $\widetilde{\omega}=\pi^{*}\omega$ is a harmonic $1$-form on $\widetilde{\Sigma}$. Moreover, $\widetilde{\omega}^{\sharp}$ coincides with the natural lift of $\omega^{\sharp}$. Therefore, by Proposition \ref{prop:key2},  we obtain ${\rm Tr}_{\wedge^{2}\fg} q_{\omega^{\sharp}}=\frac{1}{2}\cdot {\rm Tr}_{\wedge^{2}\fg}q_{\widetilde{\omega}^{\sharp}}=0$ for any harmonic $1$-form $\omega$ on $\Sigma$. Thus, by Proposition \ref{prop:ind}, we obtain an affine bound ${\rm index}(\Sigma)\geq C(b_{1}(\Sigma)-{\rm dim}\mathcal{H}_{0})$, where $\mathcal{H}_{0}=\{\omega\in \mathcal{H}\mid \mathcal{J}_{\Sigma}(\Phi_{\omega^{\sharp},v})=0\ \forall v\in \wedge^{2}\fg\}$. Since $ \mathcal{J}_{\Sigma}(\Phi_{\omega^{\sharp},v})=\mathcal{J}_{\widetilde{\Sigma}}(\varphi_{\widetilde{\omega}^{\sharp},v})$,  we see $\widetilde{\omega}\in \widetilde{\mathcal{H}}_{0}:=\{\omega\in \mathcal{H}_{\widetilde{\Sigma}}\mid\mathcal{J}_{\widetilde{\Sigma}}(\varphi_{\omega^{\sharp}, v})=0\,\forall v\in \wedge^{2}\fg \}$ if $\omega\in \mathcal{H}_{0}$, namely, the pull-back induces a linear map $\pi^{*}: \mathcal{H}_{0}\to \widetilde{\mathcal{H}}_{0}$. Since ${\rm Ker}\pi^{*}=\{0\}$, we obtain ${\rm dim}\mathcal{H}_{0}\leq {\rm dim}\widetilde{\mathcal{H}}_{0}\leq {\rm dim}\widetilde{\mathcal{H}}_{1}\leq 2n-3$ by Proposition \ref{prop:key4}, and hence, we obtain an affine bound \eqref{eq:ab}.  Therefore if $\Sigma$ is unstable, we obtain  \eqref{eq:ab2} as well as the case of two-sided hypersurface. This completes the proof. 
\end{proof}

\appendix

\section{Relation to previous methods}

\subsection{The case of $\R^{n+k}$}\label{A:Euc}

In Theorem \ref{thm:tr}, we derive a trace formula for a Riemannian manifold $M$ isometrically immersed into a semi-simple Riemannian symmetric space $N$. 
A corresponding trace formula in the case when $N=\R^{m}$ has been derived by Ambrozio-Carlotto-Sharp \cite[Proposition 2]{ACS}.  In this appendix, we remark that our argument given in Section \ref{sec:tr2} can be adapted to the case when $N=\R^{n+k}$, and their formula is recovered by a slight modification.

The identity component of the isometry group of $\R^{n+k}$ is given by a semi-direct product $G=SO(n+k)\ltimes \R^{n+k}$. Notice that $G$ is not a semi-simple Lie group (and hence, we cannot use the results given in the previous subsection directly). More precisely, $G$ is expressed by
\[
G=\Big{\{}
\begin{pmatrix}
A & b\\
0 & 1
\end{pmatrix}
\mid 
A\in SO(n+k),\ b\in \R^{n+k}\Big{\}}\subset GL(n+k+1,\R),
\]
and 
$G$ acts on $\R^{n+k}$ by 
\[
\begin{pmatrix}
A & b\\
0 & 1
\end{pmatrix}
\begin{pmatrix}
x\\
1
\end{pmatrix}
=
\begin{pmatrix}
Ax+b\\
1
\end{pmatrix},
\quad 
x\in \R^{n+k}.
\]
The Lie algebra of $\fg$ is given by
\[
\fg=\Big{\{}
\begin{pmatrix}
X & y\\
0 & 0
\end{pmatrix}
\mid 
X\in \mathfrak{o}(n+k),\ y\in \R^{n+k}\Big{\}}\subset \mathfrak{gl}(n+k+1,\R).
\]

Fix arbitrary $x\in \R^{n+k}$. Then the geodesic symmetry $s_{x}$ at $x$ is given by $s_{x}=\begin{pmatrix} -I_{n+k} & 2x\\ 0 & 1\end{pmatrix}\in G$. Thus the differential at $e\in G$ of the involution $\sigma_{x}$ is expressed by
\[
d\sigma_{x}(\begin{pmatrix}
X & y\\
0 & 0
\end{pmatrix}
)=s_{x}
\begin{pmatrix}
X & y\\
0 & 0
\end{pmatrix}
s_{x}=
\begin{pmatrix}
X & -y-X(2x)\\
0 & 0
\end{pmatrix}.
\]
Therefore, the eigendecomposition  $\fg=\fk_{x}\oplus \fn_{x}$ is given by
\[
\fk_{x}=\Big{\{}
\begin{pmatrix}
X & -Xx\\
0 & 0
\end{pmatrix}
\mid 
X\in \mathfrak{o}(n+k)\Big{\}},\quad 
\fn_{x}=\Big{\{}
\begin{pmatrix}
O & y\\
0 & 0
\end{pmatrix}
\mid 
y\in \R^{n+k}\Big{\}}.
\]
Notice that $\fn_{x}$ does not depend on the choice of $x\in N=\R^{n+k}$ and thus, we denote $\fn=\fn_{x}$.  
We define the standard inner product $\langle,\rangle$ on $\fn\simeq \R^{n+k}$, and extend it to $\wedge^{2}\fn$.

Suppose $M$ is isometrically immersed into $N=\R^{n+k}$ and let $\Sigma$ be a closed two-sided minimal hypersurface in $M$. Similar to Section \ref{sec:tr2}, we  fix a tangent vector field $T\in \Gamma(T\Sigma)$, and for each $v\in \wedge^{2}\fn$ we consider a smooth function 
\[
\widehat{\varphi}_{T, v}:=\langle \nu\wedge T, \Pi(v)\rangle,
\]
where $\nu$ is a unit normal vector field of $\Sigma$ in $\R^{n+k}$ and $\Pi(X\wedge Y)=X^{\dagger}\wedge Y^{\dagger}$ for $X,Y\in \fn$, and we linearly extend the map to whole $\wedge^{2}\fn$.
 It should be noted that, for $Y=\begin{pmatrix}
O & y\\
0 & 0
\end{pmatrix}\in \fn$, the Killing vector field generated by $Y$  is given by $Y^{\dagger}=y$, i.e. $Y^{\dagger}$ is regarded as a parallel vector field on $\R^{n+k}$.  Therefore, for an orthonormal basis $\{\theta_{i}\}_{i=1}^{n+k}$ of $\fn\simeq \R^{n+k}$, we may write
\[
\widehat{\varphi}_{T, \theta_{i}\wedge \theta_{j}}=\langle \nu\wedge T, \theta_{i}\wedge \theta_{j}\rangle.
\]
This recovers the test function considered in  \cite[Subsection 3.2]{ACS}.

Now, we define a quadratic form $\widehat{q}_{T}$ on $\wedge^{2}\fn$ by 
\begin{align*}
\widehat{q}_{T}(v,w):=Q(\widehat{\varphi}_{T,v}, \widehat{\varphi}_{T,w})
\end{align*}
and we consider the trace of $\widehat{q}_{T}$ over $\wedge^{2}\fn$. 
 We remark that we take the trace over $\wedge^{2}\fn$ (not over $\wedge^{2}\fg$. Compare with the semi-simple case \eqref{eq:trtr}). When $N=G/K_{p}$ is a semi-simple RSS,  $\fn_{p}$ depends on $p$ in general, and hence, ${\rm Tr}_{\wedge^{2}\fn_{p}} q_{T}$ is as well. This is a reason why we take the trace over $\wedge^{2}\fg$ when $N$ is semi-simple.

If $\Sigma$ is a (two-sided) closed minimal hypersurface in $M^{n}$, the computation given in Proposition \ref{prop:tr} is worked without any modification,  and if $T$ is the dual vector of a harmonic 1-form $\omega$, then we obtain 
\begin{align}\label{eq:acs}
{\rm Tr}_{\wedge^{2}\fn} \widehat{q}_{\omega^\sharp}
&=\int_{\Sigma}\sum_{\mu=1}^{n-1}\Big(|B(e_{\mu},\omega^{\sharp})|^{2}+|B(e_{\mu},\nu)|^{2}|\omega|^{2}\Big)
-\Big(\sum_{\mu=1}^{n-1}\langle R^{M}(\omega^{\sharp},e_{\mu})e_{\mu}, \omega^{\sharp}\rangle+{\rm Ric}^{M}(\nu,\nu)|\omega|^{2}\Big)
\end{align}
as $\overline{R}=0$ in \eqref{eq:tr1}. This recovers the trace formula given in \cite[Proposition 2]{ACS}. Note that the another trace formula given in \cite[Proposition 1]{ACS} is also recovered by a similar way. Moreover, by using Proposition \ref{prop:ind}-(i), we obtain the following.
\begin{theorem}[\cite{ACS}]
Suppose $M^{n}$ is isometrically immersed into $\R^{n+k}$ and let $\Sigma$ be a closed minimal hypersurface in $M$. If ${\rm Tr}_{\wedge^{2}\fn} \widehat{q}_{\omega^\sharp}<0$ for any non-trivial harmonic 1-form $\omega$ on $\Sigma$, then we have
\[
{\rm index}(\Sigma)\geq \frac{2}{(n+k)(n+k-1)}b_{1}(\Sigma).
\]
\end{theorem}
See \cite{ACS, GMR} for several examples of $M$ satisfying that ${\rm Tr}_{\wedge^{2}\fn} \widehat{q}_{\omega^\sharp}<0$ for any harmonic $1$-form $\omega$ on any closed minimal hypersurface in $M$.

\subsection{The method of Mendes-Radeschi}\label{A:MR}
In \cite{MR}, Mendes-Radeschi generalizes the previous method by using the notion of virtual immersion. According to \cite{MR}, we briefly summarize their method,  and as a consequence, we show that they use the same test function given in the present paper.  

Let $(M,g)$ be a Riemannian manifold, and $V$ be a finite-dimensional real vector space equipped with an inner product $\langle,\rangle_{V}$.  The {\it virtual immersion} of $M$ is a $V$-valued $1$-form $\Omega: TM\to V$ satisfying  the following two-conditions:
\begin{enumerate}
\item $\langle \Omega_{p}(X), \Omega_{p}(Y)\rangle_{V}=g_{p}(X,Y)$ for any $p\in M$ and $X,Y\in T_{p}M$.
\item $\langle (d\Omega)_{p}(X,Y),\Omega_{p}(Z)\rangle_{V}=0$ for any $p\in M$ and $X,Y, Z\in T_{p}M$.
\end{enumerate}

If $M$ admits  a virtual immersion $\Omega$, we obtain an immersion of the tangent bundle $TM$ to the trivial bundle $M\times V$ by the map $(p,X)\mapsto (p, \Omega_{p}(X))$ where $X\in T_{p}M$. This immersion is isometric in the sense of (i), and moreover the above two conditions imply that the natural flat connection $D$ on $M\times V$ induces the Levi-Civita connection $\nabla^{M}$ on $TM\subset M\times V$ (See \cite{MR} for details). This notion generalizes the situation of isometric immersion of $M$ into the Euclidean space $V=\R^{n+k}$. Namely, if $F: M\to V=\R^{n+k}$ is an isometric immersion, then  $F^{*}TV$ is a trivial bundle over $M$ and we have a natural immersion $TM\to F^{*}TV\simeq M\times V$, and in this case, the virtual immersion is given by $dF: TM\to V=\R^{n+k}$.
For a virtual immersion $\Omega: TM\to V$, we can define the second fundamental form $\rm{I\hspace{-.01em}I}$ of $\Omega$ by using the flat connection $D$ similar to the usual second fundamental form. Moreover, the formulas of Weingarten, Gauss, Ricci and Codazzi (relative to $\rm{I\hspace{-.01em}I}$ and $D$) also hold for the virtual immersion $\Omega$ analogous to the isometric immersion into the Euclidean space \cite[Proposition 7]{MR}.  However, the second fundamental form $\rm{I\hspace{-.01em}I}$ is not symmetric in general. 

Let $\Sigma$ be a two-sided closed minimal hypersurface in $M$, and we suppose $M$ admits a virtual immersion $\Omega: TM\to V$. Then, Mendes-Radeschi considered the following test variation along $\Sigma$ (See \cite[eq.(6)]{MR}. Note that we change the sign);
\begin{align}\label{eq:MR}
\Phi: \mathcal{H}\times \wedge^{2}V\to \Gamma(\nu\Sigma),\quad 
\Phi_{\omega, \alpha}:=\langle \Omega(\nu)\wedge \Omega(\omega^{\sharp}),\alpha\rangle_{\wedge^{2}V}\cdot \nu,
\end{align}
where $\nu$ is a unit normal vector field along $\Sigma$, $\omega$ is a harmonic $1$-form on $\Sigma$ and $\alpha\in \wedge^{2}V$. Note that if $F:M\to V=\R^{n+k}$ is an isometric immersion and $\Omega=dF$, then the test variation coincides with the one considered in \cite{ACS} (see also Appendix \ref{A:Euc}). Then, by applying the tracing method for $\Phi$, they derived a trace formula in terms of the second fundamental form $\rm{I\hspace{-.01em}I}$ of $\Omega_{0}$ and the curvature tensor of $M$ (\cite[eq.(5) and (7)]{MR}).

Now, let us assume that $M$ is a compact Riemannian symmetric space. 
In \cite{MR}, Mendes-Radeschi showed that any compact RSS admits a natural virtual immersion described as follows. Let $M=G/K$ be a compact RSS, and denote the associated canonical decomposition by $\fg=\fk\oplus \fm$. We define an ${\rm Ad}(G)$-invariant inner product $\langle, \rangle_{\fg}$ on $\fg$. Since the isotropy subgroup $K$ acts on $\fm$ via the restriction of adjoint representation, we have a homogeneous fiber bundle $G\times_{K}\fm$ on $M=G/K$ associated to the principal bundle $G\to G/K$. More precisely, $G\times_{K}\fm:=G\times \fm/\sim$, where the equivalent relation is defined by the $K$-action $k\cdot (g, X):=(gk^{-1}, {\rm Ad}(k)X)$. $G\times_{K}\fm$ is identified with the tangent bundle $TM$ by the map 
\[
G\times_{K}\fm\to TM,\quad [g, X]\mapsto dg_{o}(X),
\]
where $o=[e]\in G/K$ is the origin of $G/K$. Under this identification, we define a $\fg$-valued $1$-form on $M$ by 
\[
\Omega_{0}: G\times_{K}\fm\to \fg,\quad [g,X]\mapsto {\rm Ad}(g)X.
\]
Then, it turns out that $\Omega_{0}$ is a virtual immersion of $M$ (\cite[Lemma 11]{MR}). Moreover, the second fundamental form $\rm{I\hspace{-.01em}I}$ of $\Omega_{0}$ is skew-symmetric, and this property actually characterizes the symmetric space (\cite[Theorem B]{MR}).  

Our claim here is the following.
\begin{lemma}
Suppose $M$ is semi-simple, and $\langle,\rangle_{\fg}$ is defined by the negative-Killing form.
By taking the specific virtual immersion $\Omega_{0}$, the test function given in \eqref{eq:MR} is written as
\[
\langle \Omega_{0}(\nu)\wedge \Omega_{0}(\omega^{\sharp}), \theta\wedge \xi\rangle_{\wedge^{2}\fg}=\langle \nu\wedge \omega^{\sharp}, \theta^{\dagger}\wedge\xi^{\dagger}\rangle
\]
for any $\theta, \xi\in \fg$.

\end{lemma}
\begin{proof}
As mentioned in \cite[Remark 13]{MR}, under the identification $TM\simeq G\times_{K}\fm$, the Killing vector field $\theta^{\dagger}_{[g]}$ corresponds to $[g, ({\rm Ad}(g^{-1})\theta)_{\fm}]$ for any $g\in G$ and $\theta\in \fg$, where $(\cdot)_{\fm}$ means the orthogonal projection to $\fm$. Indeed,  we see
\begin{align*}
dg^{-1}_{[g]}(\theta^{\dagger})=\frac{d}{dt} g^{-1}{\rm exp}(t\theta) g\cdot o\Big{|}_{t=0}=\frac{d}{dt} {\rm exp}(t{\rm Ad}(g^{-1})\theta)\cdot o\Big{|}_{t=0}=({\rm Ad}(g^{-1})\theta)_{\fm},
\end{align*}
where we used the identification $\fm\simeq T_{o}M$. This shows $\theta^{\dagger}_{[g]}=[g, ({\rm Ad}(g^{-1})\theta)_{\fm}]$ under the identification $TM\simeq G\times_{K}\fm$.  Thus, by definition of $\Omega_{0}$, we have 
\begin{align}\label{eq:orel}
\Omega_{0}(\theta^{\dagger}_{[g]})={\rm Ad}(g)({\rm Ad}(g^{-1})\theta)_{\fm}=\theta-{\rm Ad}(g)({\rm Ad}(g^{-1})\theta)_{\fk}.
\end{align}

We take arbitrary vector field $Z\in \Gamma(TM)$. Then, by using \eqref{eq:orel}, we have
\begin{align*}
\langle \Omega_{0}(Z), \theta\rangle_{\fg}&=\langle \Omega_{0}(Z), \Omega_{0}(\theta^{\dagger})\rangle_{\fg}-\langle \Omega_{0}(Z), {\rm Ad}(g)({\rm Ad}(g^{-1})\theta)_{\fk}\rangle_{\fg}\\
&=\langle Z, \theta^{\dagger}\rangle-\langle \Omega_{0}(Z), {\rm Ad}(g)({\rm Ad}(g^{-1})\theta)_{\fk}\rangle_{\fg},
\end{align*}
where we used the condition (i) in the definition of virtual immersion.
For each point $[g]\in M$, there exists a vector $X_{Z}\in \fg$ satisfying that $Z_{[g]}=(X_{Z}^{\dagger})_{[g]}$. Then, by \eqref{eq:orel}, we see
\begin{align*}
\langle \Omega_{0}(Z), {\rm Ad}(g)({\rm Ad}(g^{-1})\theta)_{\fk}\rangle_{\fg}
&=\langle {\rm Ad}(g)({\rm Ad}(g)^{-1}X_{Z})_{\fm}, {\rm Ad}(g)({\rm Ad}(g^{-1})\theta)_{\fk}\rangle_{\fg}\\
&=\langle ({\rm Ad}(g)^{-1}X_{Z})_{\fm}, ({\rm Ad}(g^{-1})\theta)_{\fk}\rangle_{\fg}=0
\end{align*}
since $\langle,\rangle_{\fg}$ is ${\rm Ad}(G)$-invariant and the canonical decomposition $\fg=\fk\oplus \fm$ is orthogonal with respect to the (negative) Killing form (see Lemma \ref{lem:key2}). Therefore, we obtain $\langle \Omega_{0}(Z), \theta\rangle_{\fg}=\langle Z,\theta^{\dagger}\rangle$ for any $Z\in \Gamma(TM)$, and this implies the desired conclusion.
\end{proof}

Hence, the test function considered in \cite{MR} coincides with our test function given in Section \ref{sec:tr2} in the case when $M$ is itself a compact semi-simple RSS. Note that, however, our setting considered in Section \ref{sec:tr2} is more general, and the expression of the trace formula is different from \cite{MR}.

\section{Local smoothness of simultaneous eigenvectors} \label{A:smooth}
In this  Appendix, we prove the following technical result used in the proof of Proposition \ref{prop:key4}. 

\begin{proposition}\label{prop:key5}
Let $(\Sigma^{m},g)$ be a smooth Riemannian manifold, and  $\alpha,\beta$ be smooth symmetric $(0,2)$-tensor fields on $\Sigma$. Suppose furthermore the $(1,1)$-tensor fields $A:=g^{-1}\alpha$ and $B:=g^{-1}\beta$ commute with each other, i.e. $[A,B]=0$ on $\Sigma$. Then, there exits an open subset $U\subset \Sigma$ and a smooth orthonormal frame $\{E_{\mu}\}_{\mu=1}^{m}$ on $U$ such that $A$ and $B$ are simultaneously diagonalized by $\{E_{\mu}\}_{\mu=1}^{m}$ at every point $p\in U$.
\end{proposition}

The proof relies on the result by Singley \cite{Sing}.  We briefly summarize his result before the proof. 

Let $G$ and $G'$ be smooth symmetric $(0, 2)$-tensor fields on a smooth $m$-dimensional manifold $\Sigma$. For each $p\in \Sigma$, $G$ defines a linear map $T_p\Sigma\to T_p^*\Sigma$, $v\mapsto G(v, \cdot)$ 
which is also denoted by $G$. We define the linear map $G':T_p\Sigma\to T^*_p\Sigma$ as well. If $G$ is positive definite, then $G$ is invertible, and we obtain a linear map $(G^{-1}G')(p):T_p\Sigma\to T_p\Sigma$. For example, if $\Sigma^{m}$ is a hypersurface  in a Riemannian manifold $(M^{m+1}, g)$, we may choose the symmetric $(0,2)$-tensors as the induced metric $g$ and $\widetilde{A}(X,Y):=\langle A(X,Y),\nu\rangle$, where $A$ is the second fundamental form of $\Sigma^{m}$. In this case, $G^{-1}G'=g^{-1}\widetilde{A}$ coincides with the shape operator $S_{\nu}$ of $\Sigma$. Since both $G$ and $G'$ are symmetric and $G$ is invertible,  we see that $G^{-1}G'$ is diagonalizable with real eigenvalues $\{\kappa_{\mu}\}_{\mu=1}^{m}$ and a corresponding orthonormal eigenbasis $\{e_\mu\}_{\mu=1}^{m}$ of $T_{p}\Sigma$. Note that, in general, there is no guarantee that the eigenvalues and eigenvectors are smooth on whole of $\Sigma$. However, it turns out to be true on an open dense subset on $\Sigma$:

\begin{lemma}[Theorem 1, 2 and 3 in \cite{Sing}]\label{lem:Sing}
  After the re-ordering of $\{\kappa_\mu\}_{\mu=1}^{m}$ by $\kappa_1(p)\leq \dots \leq \kappa_{m}(p)$, we have the following: 
  \begin{itemize}
    \item[(1)] $\kappa_{\mu}$ is continuous on whole of $\Sigma$ for any $\mu=1,\ldots, m$.
    \item[(2)] There is an open dense subset $U\subset \Sigma$ on which $\kappa_\mu$ is $C^\infty$ for any $\mu=1,\ldots, m$.
    \item[(3)] For any point $p\in U$, there is a neighborhood $U_p$ of $p$ and a $C^\infty$ orthonormal frame $\{E_\mu\}_{\mu=1}^{m}$ on $U_p$ which diagonalizes $G^{-1}G'(q)$ for each $q\in U_p$. Moreover, the multiplicities of eigenvalues are constant on $U_{p}$.
  \end{itemize}
\end{lemma}
\begin{remark}
{\rm  In \cite{Sing}, the open dense subset $U\subset \Sigma$ is explicitly given by considering the multiplicity of the eigenvalues of $G^{-1}G'$. However we do not need the explicit form of $U$ in this paper. }
\end{remark}

By using this lemma, we shall give a proof of Proposition \ref{prop:key5}.

\begin{proof}[Proof of Proposition \ref{prop:key5}]
By Lemma \ref{lem:Sing}, there exists a point $p_{0}\in \Sigma$ and an open neighborhood $U_{p_{0}}$ of $p_{0}$ such that $A=g^{-1}\alpha$ is diagonalized by a smooth orthonormal frame $\{e_{\mu}\}_{\mu=1}^{m}$ on $U_{p_{0}}$. Moreover, we may assume the multiplicities of eigenvalues of $A$ are constant on $U_{p_{0}}$ and  $A$ is expressed by the diagonal matrix  $A={\rm diag}(\lambda_{1}I_{m_{1}},\cdots, \lambda_{r}I_{m_{r}})$   on $U_{p_{0}}$ with respect to the local frame $\{e_{\mu}\}_{\mu=1}^{m}$, 
where $\lambda_{1}<\lambda_{2}<\cdots<\lambda_{r}$ are the distinct eigenvalues of $A$,  $m_{i}$ is the multiplicity of $\lambda_{i}$ and $I_{m_{i}}$ is the identity matrix of order $m_{i}$.

 Since $[A,B]=0$, each eigenspace of $A$ is invariant by the left multiplication of $B$, and this implies that $B$ is expressed by 
 the block matrix $B={\rm diag}(B_{1},\ldots, B_{r})$   on $U_{p_{0}}$ with respect to the local frame $\{e_{\mu}\}_{\mu=1}^{m}$, where $B_{i}$ is some square matrix of order $m_{i}$. Note that each $B_{i}$ is a symmetric matrix since so is $B$.
  We consider a  symmetric matrix defined by $\widetilde{B}_{1}:={\rm diag}(B_{1}, I_{m_{2}},\ldots, I_{m_{r}})$ which represents some $(1,1)$-tensor field $g^{-1}\widetilde{\beta}_{1}'$ on $U_{p_{0}}$, where we regard $U_{p_{0}}$ as a Riemannian manifold by the induced metric.  
  By Lemma \ref{lem:Sing}, there exists a point $p_{1}\in U_{p_{0}}$ such that there is an open neighborhood $U_{p_{1}}\subset U_{p_{0}}$ of $p_{1}$ and a smooth orthonormal frame $\{e_{\mu}^{(1)}\}_{\mu=1}^{m}$ which diagonalizes $\widetilde{B}_{1}$ and the multiplicities of eigenvalues of $\widetilde{B}_{1}$ are constant on $U_{p_{1}}$.  Since $\widetilde{B}_{1}$ has eigenvalue $1$, we denote the $1$-eigenspace of $\widetilde{B}_{1}$ by $V^{(1)}$ and set $k_{1}:={\rm dim}V^{(1)}$ which is constant on $U_{p_{1}}$. Note that $k_{1}\geq m_{2}+\cdots +m_{r}=m-m_{1}$. 
  
  By taking a re-ordering if necessary, we may assume that $V^{(1)}$ is spanned by $\{e_{i}^{(1)}\}_{i=m-k_{1}+1}^{m}$.
Since $e_{l}\in V^{(1)}$ for any $l>m_{1}$,  we see that $W^{(1)}:={\rm span}\{e_{m_{1}+1},\ldots, e_{m}\}$ is an $(m-m_{1})$-dimensional linear  subspace of $V^{(1)}$. If $V^{(1)}\neq W^{(1)}$, we can find $k_{1}-(m-m_{1})$ vectors $e^{(1)}_{i_{1}},\ldots, e^{(1)}_{i_{k_{1}-(m-m_{1})}}$ in $\{e_{i}^{(1)}\}_{i=m-k_{1}+1}^{m}$ so that $\{e_{i_{1}}^{(1)},\ldots, e_{i_{k_{1}-(m-m_{1})}}^{(1)}, e_{m_{1}+1},\dots, e_{m}\}$ becomes a basis of $V^{(1)}$. By the Gram--Schmidt orthonormalization, we obtain an orthonormal basis  of $V^{(1)}$ of the form $\{e_{i_{1}}^{'},\ldots, e_{i_{k_{1}-(m-m_{1})}}^{'}, e_{m_{1}+1},\dots, e_{m}\}$. Now, we put
\[
Q_{1}:=(e_{1}^{(1)},\ldots,e_{m-k_{1}}^{(1)}, e_{i_{1}}^{'},\ldots, e_{i_{k_{1}-(m-m_{1})}}^{'}, e_{m_{1}+1},\dots, e_{m}).
\]
By construction, $Q_{1}$ consists of eigenvectors of $\widetilde{B}_{1}$, and each component of  $Q_{1}$ is a smooth function on $U_{p_{1}}$. Moreover, $Q_{1}$ is expressed by ${\rm diag}(Q_{1}', I_{m_{2}},\ldots, I_{m_{r}})$ for some square matrix $Q_{1}'$ of order $m_{1}$ (since the matrix is expressed by using $\{e_{\mu}\}_{\mu=1}^{m}$), and we have $Q_{1}^{-1}\widetilde{B}_{1}Q_{1}={\rm diag}(B_{1}', I_{m_{2}},\ldots, I_{m_{r}})$, where $B_{1}'$ is the diagonalized matrix of $B_{1}$.

  Next, we consider a symmetric matrix $\widetilde{B}_{2}={\rm diag}(I_{m_{1}}, B_{2}, I_{m_{3}}, \ldots, I_{m_{r}})$ and by a similar argument, we see that there is a point $p_{2}\in U_{p_{1}}$, an open neighborhood $U_{p_{2}}\subset U_{p_{1}}$ and a matrix $Q_{2}$ whose components are smooth function on $U_{p_{2}}$ satisfying that $Q_{2}={\rm diag}(I_{m_{1}}, Q_{2}', I_{m_{3}},\ldots, I_{m_{r}})$ and  $Q_{2}^{-1}\widetilde{B}_{2}Q_{2}={\rm diag}(I_{m_{1}}, B_{2}', I_{m_{3}},\ldots, I_{m_{r}})$, where $B_{2}'$ is the diagonalized matrix of $B_{2}$.

  By repeating the argument, we obtain a decreasing sequence of open subsets $U_{p_{0}}\supset U_{p_{1}}\supset \cdots \supset U_{p_{r}}$ and matrices $Q_{1},\ldots, Q_{r}$ defined on $U_{p_{r}}$. We set $R:=Q_{1}Q_{2}\cdots Q_{r}$. Then, we see that $R^{-1}AR$ and $R^{-1}BR$ are diagonal matrices by construction of $Q_{i}$. Thus we put $E_{\mu}:=R(e_{\mu})$ for each $\mu=1,\ldots, m$ and then, since each $Q_{i}$ is an orthogonal matrix and has smooth components,  $\{E_{\mu}\}_{\mu=1}^{m}$ is a smooth orthonormal frame defined on $U_{p_{r}}$ which simultaneously diagonalizes $A$ and $B$. This completes the proof.
\end{proof}

\subsection*{Acknowledgements}
The authors would like to thank the referees for their kind comments and valuable suggestions. T.K. is supported by JSPS KAKENHI Grant Number 23K03122. K.K. is supported by JSPS KAKENHI Grant Number 23K03105.

%

\end{document}